\documentclass[a4paper,abstracton,10pt]{./siamart1116}

\usepackage{bm}
\usepackage[a4paper]{geometry}
\usepackage{booktabs}
\usepackage{stmaryrd}
\usepackage{pgfplots}
\usepackage{hyperref}
\usepackage{calc}
\usepackage{enumitem}
\usepackage{ulem}
\usepackage{mathtools}
\usepackage{microtype}
\pgfplotsset{compat=newest}
\pgfplotsset{plot coordinates/math parser=false}

\usepackage{pifont}
\usepackage[numbers]{natbib}
\usepackage{color}

\definecolor{darkgreen}{rgb}{0.1,0.5,0.1}
\definecolor{darkblue}{rgb}{0.1,0.1,0.9}

\newcommand{\spinv}{\mathsf{spinv}}

\newcommand{\E}{\mathbb{E}}

\newcommand{\ginvset}{\mathcal{G}}

\newcommand{\T}{\top}

\newcommand{\ginv}[2]{\ensuremath{\mathsf{ginv}_{#1}(#2)}}

\newcommand{\ind}[1]{\mathbb{I}_{#1}}

\newcommand{\proj}{\ensuremath{\mathsf{proj}}}

\newcommand{\prob}{\mathbb{P}}

\newcommand{\dist}{\mathsf{dist}}

\newcommand{\erfc}{\mathrm{erfc}}
\newcommand{\setS}{{\cal S}}
\newcommand{\var}{\mathsf{Var}}

\DeclareMathOperator*{\argmin}{arg\,min}

\usepackage{amsmath}
\usepackage{amssymb}
\usepackage{latexsym}
\usepackage{verbatim}
\usepackage{graphicx}
\usepackage{mdframed}

\definecolor{shade}{rgb}{1,1,1}

\newmdtheoremenv [backgroundcolor=shade, %
innertopmargin = -4pt , %
innerbottommargin =2pt , %
innerleftmargin = 1pt , %
innerrightmargin = 1pt, %
splittopskip = \topskip, %
skipbelow= 6pt, %
skipabove=6pt, %
topline=false,bottomline=false,leftline=false,rightline=false,]{example}{Example}[section]

\newmdtheoremenv [backgroundcolor=shade, %
innertopmargin = -4pt , %
innerbottommargin =2pt , %
innerleftmargin = 1pt , %
innerrightmargin = 1pt, %
splittopskip = \topskip, %
skipbelow= 6pt, %
skipabove=6pt, %
topline=false,bottomline=false,leftline=false,rightline=false,]{remark}{Remark}[section]

{\begin{list}               
    {$\bullet$ \hfill}{
        \setlength{\leftmargin}{\parindent}
        \setlength{\parsep}{0.04\baselineskip}
        \setlength{\itemsep}{0.5\parsep}
        \setlength{\labelwidth}{\leftmargin}
        \setlength{\labelsep}{0em}}
    }
{\end{list}}

\providecommand{\cref}[1]{Chapter~\ref{chap:#1}}

\providecommand{\R}{\ensuremath{\mathbb{R}}}
\providecommand{\C}{\ensuremath{\mathbb{C}}}
\providecommand{\N}{\ensuremath{\mathbb{N}}}

\providecommand{\abs}[1]{\left|#1\right|}
\providecommand{\norm}[1]{\left\lVert#1\right\rVert}

\providecommand{\set}[1]{\left\lbrace#1\right\rbrace}

\providecommand{\bydef}{\overset{\mathrm{def}}{=}}

\providecommand{\parder}[2]{{\partial{#1} \over \partial{#2}}}
\providecommand{\parderr}[2]{{\partial^2{#1} \over \partial{#2}^2}}

\providecommand{\di}{\ensuremath{\text{d}}}
\providecommand{\e}{\ensuremath{\mathrm{e}}}

\renewcommand{\vec}[1]{\ensuremath{\mathbf{#1}}}
\providecommand{\mat}[1]{\ensuremath{\mathbf{#1}}}

\providecommand{\wt}[1]{\ensuremath{\widetilde{#1}}}

\providecommand{\mA}{\mat{A}}

\providecommand{\mI}{\mat{I}}

\providecommand{\mM}{\mat{M}}

\providecommand{\mW}{\mat{W}}

 \providecommand{\mG}{\mat{G}}

\providecommand{\mX}{\mat{X}}

 \providecommand{\vb}{\vec{b}}
 
\providecommand{\ve}{\vec{e}} 
 
\providecommand{\vg}{\vec{g}}
\providecommand{\vh}{\vec{h}} 
  
 \providecommand{\vp}{\vec{p}}

\providecommand{\vu}{\vec{u}} \providecommand{\vw}{\vec{w}}

\providecommand{\vx}{\vec{x}} \providecommand{\vy}{\vec{y}}
\providecommand{\vz}{\vec{z}} 
 \providecommand{\vzero}{\vec{0}}
 \providecommand{\vv}{\vec{v}}

\renewcommand{\theequation}{\thesection.\arabic{equation}}
\renewcommand{\thefigure}{\thesection.\arabic{figure}}
\renewcommand{\thetable}{\thesection.\arabic{table}}

\numberwithin{theorem}{section}

\makeatletter
\newcommand{\inlineitem}[1][]{%
\ifnum\enit@type=\tw@
    {\descriptionlabel{#1}}
  \hspace{\labelsep}%
\else
  \ifnum\enit@type=\z@
       \refstepcounter{\@listctr}\fi
    \quad\@itemlabel\hspace{\labelsep}%
\fi}
\makeatother

\title{Concentration of the Frobenius Norm of \\Generalized Matrix Inverses}
\author{Ivan Dokmani\'c
\thanks{I. Dokmanic is with Dept of Electrical and Computer Engineering, University of Illinois at Urbana-Champaign} 
\and 
R\'emi Gribonval
\thanks{R. Gribonval is with Univ Rennes, Inria, CNRS, IRISA\protect\\ \noindent\makebox[\linewidth]{\rule{\textwidth}{0.4pt}}}
}
\date{}
\begin{document}

\maketitle
\normalem


\begin{abstract}
In many applications it is useful to replace the Moore-Penrose pseudoinverse (MPP) by a different generalized inverse with more favorable properties. We may want, for example, to have many zero entries, but without giving up too much of the stability of the MPP. One way to quantify stability is by how much the Frobenius norm of a generalized inverse exceeds that of the MPP. In this paper we derive finite-size concentration bounds for the Frobenius norm of $\ell^p$-minimal general inverses of iid Gaussian matrices, with $1 \leq p \leq 2$. For $p = 1$ we prove exponential concentration of the Frobenius norm of the sparse pseudoinverse; for $p = 2$, we get a similar concentration bound for the MPP. Our proof is based on the convex Gaussian min-max theorem, but unlike previous applications which give asymptotic results, we derive finite-size bounds.
\end{abstract}

\section{Introduction}
\label{sec:introduction}

Generalized inverses are matrices that have some properties of the usual inverse of a regular square matrix. We call\footnote{We only study real matrices.} $\mX \in \R^{n \times m}$ a generalized inverse of a matrix $\mA \in \R^{m \times n}$ if $\mA \mX \mA = \mA$, and write $\ginvset(\mA)$ for the set of all such matrices $\mX$. In inverse problems it is desirable to use generalized inverses with a small Frobenius norm because the Frobenius norm controls the output mean-squared error. In this paper, we study Frobenius norms of generalized inverses that are obtained by constrained minimization of a class of matrix norms. Concretely, we look at entrywise $\ell^p$ norms for $1 \leq p \leq 2$, including the Moore-Penrose pseudoinverse (MPP) for $p = 2$ and the \emph{sparse pseudoinverse} for $p = 1$.

For $\mA \in \R^{m \times n}$, $m < n$, $1 \leq p \leq \infty$, we define the $\ell^p$-minimal generalized inverse as\footnote{For $p=1$, $\ginv{p}{\mA}$ is a singleton except for a set of matrices of measure zero; a precise statement is given in Section \ref{sec:stability}.}
\begin{align*}
  \ginv{p}{\mA} &\bydef \argmin_{\mX} \ \norm{\vec{\mX}}_{p} \ \ \text{subject to} \ \ \mX \in \ginvset(\mA),
\end{align*}
with 
\(
  \norm{\mM}_p  = \big( \sum_{ij} \abs{m_{ij}}^p \big)^{1/p}.
\)
The MPP $\mA^\dag$ is obtained by minimizing the Frobenius norm for $p = 2$, thus $\norm{\ginv{p}{\mA}}_F \geq \norm{\mA^\dag}_F$ for any $p$. Computing $\ginv{p}{\mA}$ involves solving a convex program.

Our initial motivation for this work is the sparse pseudoinverse, $p = 1$, since applying a sparse pseudoinverse requires less operations than applying a full one \cite{Dokmanic:2013bo,Li:2013cx,Casazza:ev,Krahmer:2012cy}. The sparsest generalized inverse may be formulated as
\begin{equation}
    \label{eq:intro_l0}
    \ginv{0}{\mA} \bydef \argmin_{\mX}~\norm{\mX}_{0}\ \text{subject to} \ \mX \in \ginvset(\mA),
\end{equation}
where $\norm{\, \cdot \,}_{0}$ counts the total number of non-zero entries. The non-zero count gives the naive complexity of applying $\mX$ or its adjoint to a vector. We show in Section \ref{sec:stability} that for a generic $\mA$, \eqref{eq:intro_l0} has many solutions. While some of them are easy to compute,  they correspond to poorly conditioned matrices.

To recover uniqueness and improve conditioning of sparse pseudoinverses, it is natural to try and replace $\ell^0$ by the $\ell^{1}$ norm \cite{Donoho:2006ci}. Indeed, it was shown in \cite{beyondMP-partII} that $\ginv{1}{\mA}$ provides \emph{a} minimizer of the $\ell^{0}$ norm for almost all matrices, and that this minimizer is generically unique, motivating the notation $\spinv(\mA) \bydef \ginv{1}{\mA}$. Intuitively, an $m \times n$ matrix $\mA$ with $m \leq n$ is generically rank $m$, hence $\mA \mX = \mI_{m}$ is a system of $m^2$ independent linear equations. The matrix $\mX$ has $nm$ entries, leaving us with $nm - m^2$ degrees of freedom, which one hopes will correspond to $nm - m^2$ zero entries. The main advantage of $\spinv$ over $\ell^0$ minimization is that it yields unique, well-behaved matrices. 


\subsection{Our Contributions}

In Sections~\ref{sec:stability} and \ref{sec:proof-main-concentration} we prove an exponential concentration result for the Frobenius norm of $\ell^p$-minimal pseudoinverses for iid Gaussian matrices. Specializing to $p = 1$ in Corollary \ref{cor:P1}, we show that unlike simpler strategies that yield a sparsest generalized inverse, $\ell^1$ minimization produces a well-conditioned matrix; specializing to $p = 2$ in Corollary \ref{cor:P2}, we get new results for the Frobenius norm of the MPP. Unlike previous applications of the CGMT, we give finite-size concentration bounds rather than asymptotic ``in probability'' results.

\subsection{Prior Art}

There is a one-to-one correspondence between generalized inverses of full rank matrices and dual frames. Several earlier works \cite{Casazza:ev, Krahmer:2012cy, Li:2013cx} study existence  and explicit constructions of sparse frames and sparse dual frames. Krahmer, Kutyniok, and Lemvig \cite{Krahmer:2012cy} establish sharp bounds on the sparsity of dual frames, showing that generically, for $\mA \in \C^{m \times n}$, the sparsest dual has $mn - m^2$ zeros, while Li, Liu, and Mi \cite{Li:2013cx} provide better bounds on the sparsity of dual Gabor frames. They introduce the idea of using $\ell^1$ minimization to find these dual frames, and show that under certain conditions,  $\ell^1$ minimization yields the sparsest possible dual Gabor frame. Further examples of non-canonical dual frames are given by Perraudin et al., who use convex optimization to derive dual frames with good time-frequency localization \cite{Perraudin:2014we}.

Results on finite-size concentration bounds for norms of pseudoinverses are scarce, with one notable exception being an upper bound on the probability of large deviation for the MPP \cite{Halko:2011kg} (we obtain concentration bounds for a complementary regime). On the other hand, a number of results exist for square matrices \cite{Rudelson:2008gv,Vershynin:2014jv}. The sparse pseudoinverse was previously studied in \cite{Dokmanic:2013bo}, where it was shown empirically that the minimizer is indeed a sparse matrix, and that it can be used to speed up the resolution of certain inverse problems.

Our proof relies on the \emph{convex Gaussian min-max theorem} (CGMT) \cite{Oymak:2013vm,Thrampoulidis:2016vo,Thrampoulidis:2015vf} which was previously used to quantify performance of regularized M-estimators such as the lasso \cite{Oymak:2013vm,Thrampoulidis:2016vo,Thrampoulidis:2015vf}. Many technical ideas in \cite{Thrampoulidis:2016vo,Oymak:2013vm,Thrampoulidis:2015vf} have been developed in earlier works. The CGMT can be seen as a descendant of Gordon's Gaussian min-max theorem \cite{Gordon1985}. Rudelson and Vershynin \cite{RudelsonVershynin} first recognized that Gordon's result (more precisely, its consequence known as escape through a mesh \cite{Gordon1988escape}) is a useful theoretical device to study sparse regression. Ensuing papers by Stojni\'c \cite{Stojnic2009,Stojnic2013}, Chandrasekaran et al. \cite{Chandrasekaran:2012fs}, Amelunxen et al. \cite{Amelunxen2014}, Foygel and Mackey \cite{Foygel:2014iy}, and others, give sharper analyses and study more general settings. Their techniques percolated into the work by Thrampoulidis et al. \cite{Thrampoulidis:2016vo} which we primarily refer to.
  
\section{Frobenius Norms of Generalized Inverses}
\label{sec:stability}

In this section we state our main results and prepare the proof. We first need to clear a technicality: for some $\mA$, $\spinv(\mA)$ will have multiple minimizers. This is, however, rare, as we prove in \cite{beyondMP-partII}:
\begin{theorem} \label{th:spinvmsparse}
Assume that $\mA \in \R^{m \times n}$ has columns in general position, and that $\mA$ is in general position with respect to the canonical basis vectors $\ve_{1},\ldots,\ve_{m}$. Then the sparse pseudoinverse $\spinv(\mA)$ of $\mA$ contains a single matrix whose columns are all exactly $m$-sparse.
\end{theorem}
Operationally, this means that for almost all matrices with respect to the Lebesgue measure on $\R^{m \times n}$ we will have a unique $\spinv$, which simplifies the proofs.\footnote{A careful reader will notice that the assumptions of Theorem \ref{th:spinvmsparse} forbid many types of sparse matrices. Indeed, it is known that sparse frames can have sparser duals than generic frames \cite{Krahmer:2012cy,Li:2013cx}.}

\begin{figure}[t!]
\centering
\includegraphics[width=.7\linewidth]{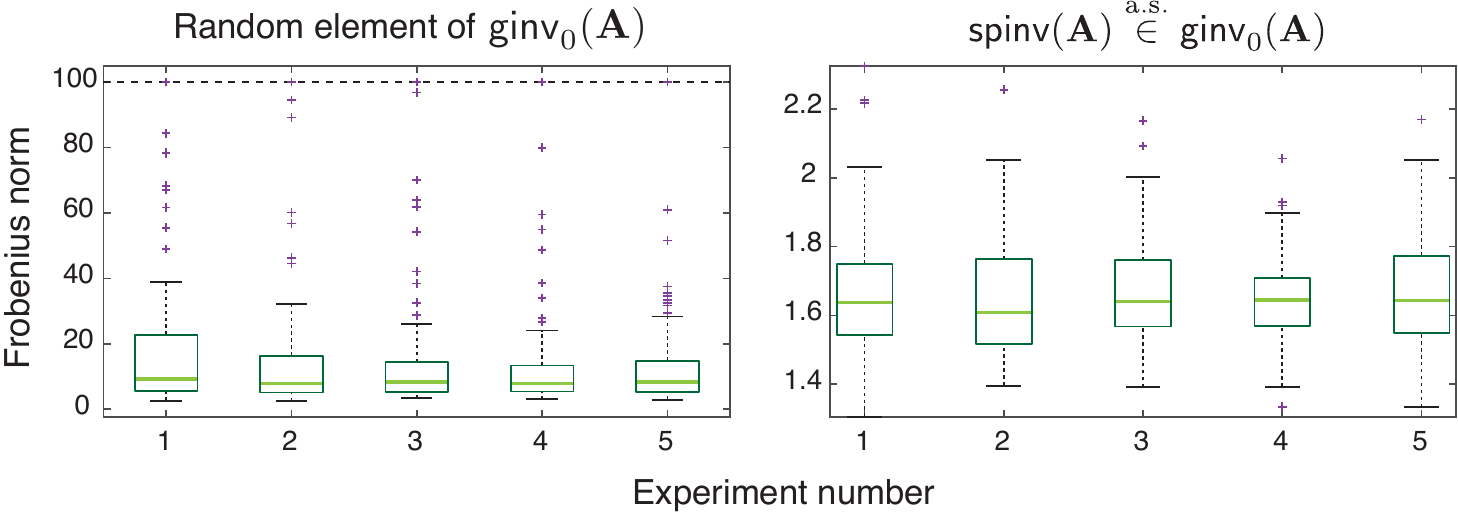}
\caption{Frobenius norm of a random element of $\ginv{0}{\mA}$ for $\mA$ a random Gaussian matrix of size $20 \times 30$ (left) and a sparse pseudoinverse $\spinv(\mA)$ (generically in $\ginv{0}{\mA}$; cf. \cite{beyondMP-partII}) of the same matrix (right). The random element of $\ginv{0}{\mA}$ is computed by selecting $m$ columns of $\mA$ at random and inverting the obtained submatrix. The plot shows results of 5 identical experiments, each consisting of generating a random $20 \times 30$ matrix and computing the two inverses 100 times. In the first experiment, the outlier norms extend to 2000 so they were clipped at 100. Green lines denote medians, and boxes denote the second and the third quartile.}
\label{fig:minor}
\end{figure}

The importance of the Frobenius norm of a generalized inverse can be motivated by considering an overdetermined system of linear equations $\vy = \mA^\T \vx + \vz$, where $\mA \in \C^{m \times n}$ has full row rank and $\vz$ is noise such that $\E \{ \vz \vz^\T \} \propto \mI_n$. Then for $\mW \in \ginvset(\mA)$ we have
\(
    \E_{\vz}[\norm{\mW \vy - \vx}_2^2] = \E_{\vz}[\norm{\mW \vz}_2^2] \, \propto\, \norm{\mW}_F^2. 
\)
Thus the mean-squared error is controlled by the Frobenius norm of $\mW$ and it is desirable to use generalized inverses with small Frobenius norms.

Note that for a rank-$m$ $\mA$, the optimization \eqref{eq:intro_l0} decouples over columns into $m$ independent $\ell^0$ problems. Even though such problems are in general NP-hard \cite{DMA97,natarajan95:_spars}, Theorem~\ref{th:spinvmsparse} says that for a generic $\mA \in \R^{m \times n}$ with $m < n$, a sparsest generalized inverse has $m^2$ non-zeros. Thus a simple way to compute \emph{a} sparsest generalized inverse (an element of $\ginv{0}{\mA}$) is to invert any $m \times m$ submatrix of $\mA$ and set the rest to zero. Unfortunately, this leads to poorly conditioned matrices, unlike computing the $\spinv$ as shown in Figure \ref{fig:minor}.

The goal of this section is to make the last statement quantitative by developing concentration results for the Frobenius norm of $\spinv(\mA)$ as well as $\ginv{p}{\mA}$ for $1 \leq p \leq 2$ when $\mA$ is an iid Gaussian random matrix. 

Our results rely on the properties of the following geometric functional:
\begin{eqnarray}
\label{eq:DefExpSquaredDistance}
D(t) = D_{p}(t;n) &\bydef& \tfrac{1}{n} \left[\E_{\vh \sim {\cal N}(0, \mI_n)} \,
\dist(\vh, \norm{\,\cdot\,}_{p^{*}} \leq t)\right]^{2}, \ \text{$1 \leq p \leq \infty$, $t \geq 0$},
\end{eqnarray}
where $\dist(\vh, \norm{\,\cdot\,}_{p^*} \leq t) \bydef \min_{\vu : \norm{\vu}_{p^*} \leq t} \norm{\vh - \vu}_2$ and $1/p + 1/p^* = 1$.
In particular, Lemma~\ref{lem:Dp_det_prop}--Property 7 in Appendix \ref{appendix:stability} tells us that $0 < D(t) < 1$ and for $0 < \delta < 1$, $\delta = D(t) -\tfrac{t}{2} D'(t)$ on $(0,\infty)$ has a unique solution on $(0, \infty)$, denoted $t^* = t^{*}_{p}(\delta, n)$. With this notation we can state our main result.

\begin{theorem}
    \label{thm:frob_of_l1}
    Let $\mA \in \R^{m \times n}$, $1 < m \leq n$, be a standard iid Gaussian matrix and $\delta \bydef (m-1) / n \in (0, 1)$. For $1 \leq p \leq 2$, define
    $t^* = t^{*}_{p}(\delta,n)$ to be the unique solution of $\delta = D(t) -\tfrac{t}{2} D'(t)$ on $(0,\infty)$ and denote 
    \[
    \alpha^* = \alpha^*_{p}(\delta, n) \bydef \sqrt{\frac{D(t^{*})}{\delta(\delta - D(t^{*}))}}. 
    \]
    Assume there exist\footnote{The existence of $\gamma(\delta)$ and $N(\delta)$ will be proved below for $p \in \{1,2\}$.} $\gamma(\delta)$ and $ N(\delta)$ such that $-t^{*}D'_{p}(t^{*};n) \geq \gamma(\delta) >0$ for all $n \geq N(\delta)$. Then for any $n \geq \max(2/(1-\delta),N(\delta))$ we have: for any $0<\epsilon \leq 1$,
    \begin{equation}
        \label{eq:stability_main_bound}
        \prob \left[ \abs{ \tfrac{n}{m}\norm{\ginv{p}{\mA}}_F^2 -  (\alpha^*)^2} \geq \epsilon(\alpha^*)^2 \right] \leq \   \frac{n}{C_{1}\epsilon}\e^{-C_{2}n\epsilon^{4}},
    \end{equation}
    where the constants $C_{1},C_{2}>0$ may depend on $\delta$ but not on $n$ or $\epsilon$.
\end{theorem}

From Theorem \ref{thm:frob_of_l1} we can derive more explicit results for the two most interesting cases: $p=1$ and $p=2$. For $p=2$ we get a result about $\norm{\mA^\dag}_F$ complementary to a known large deviation bound \cite[Proposition A.5; Theorem A.6]{Halko:2011kg} obtained by a different technique.
\begin{corollary}[$p=2$]\label{cor:P2}
  Let $\mA \in \R^{m \times n}$, $1 < m \leq n$, be a standard iid Gaussian matrix, $\delta = (m-1) / n \in (0, 1)$, and
  \(
  \alpha^* = \alpha^*_{2}(\delta, n) \bydef \sqrt{\frac{D_2(t^{*};n)}{\delta(\delta - D_2(t^{*};n))}},
  \)
  with $t^*$ being the unique solution of $\delta = D_2(t; n) -\tfrac{t}{2} D_2'(t; n)$ on $(0,\infty)$. Then there exists $N(\delta)$ such that for $n \geq N(\delta)$ we have: for any $0<\epsilon \leq 1$,
\begin{equation}
        \prob \left[ \abs{ \tfrac{n}{m}\norm{\mA^{\dagger}}_F^2 -  (\alpha^*)^2} \geq \epsilon(\alpha^*)^2 \right] \leq \   \frac{n}{C_{1}\epsilon}\e^{-C_{2}n\epsilon^{4}},
    \end{equation}
    where the constants $C_{1},C_{2}>0$ may depend on $\delta$ but not on $n$ or $\epsilon$, and $\alpha^{*}=\alpha^{*}_{2}(\delta;n)$ with
    \begin{equation}
    \lim_{n \to \infty} \alpha_{2}^{*}(\delta;n) = {1}/{\sqrt{1-\delta}}.
    \end{equation}
\end{corollary}

We remark that Corollary \ref{cor:P2} covers ``small'' deviations $(0<\epsilon \leq 1)$. In contrast, the result of \cite{Halko:2011kg} establishes\footnote{The results in \cite{Halko:2011kg} are designed for the nonasymptotic regime where the matrix is essentially square.} that $\mathbb{E}\ \tfrac{n}{m}\norm{\mA^{\dagger}}_F^2 = \tfrac{n}{n-m-1} = \tfrac{1}{1-\delta}$ and that for any $\tau \geq 1$,
\[
\prob \left[ \tfrac{n}{m}\norm{\mA^{\dagger}}_F^2 \geq \tfrac{12 n}{n-m} \tau\right] \leq 4 \tau^{-(n-m)/4} = 4 e^{-n \tfrac{(1-m/n) \log \tau}{4}}.
\]
For large $n$ we have $\tfrac{12 n}{n-m} \tau \approx \tfrac{12\tau}{1-\delta} \approx 12\tau (\alpha^{*}_{2})^{2}$ and $\tfrac{(1-m/n) \log \tau}{4} \approx \tfrac{(1-\delta) \log \tau}{4}$, hence this provides a bound for $1+\epsilon := 12 \tau \geq 12$ with an exponent $\log \tau$, of the order of $\log \epsilon$ (instead of $\epsilon^{4}$ we get for $0<\epsilon \leq 1$). Further, we also show that the probability of $\tfrac{n}{m}\norm{\mA^\dag}_F^2$ being much smaller than $(\alpha^*)^2$ is exponentially small.

The most interesting corollary is for the sparse pseudoinverse, $p = 1$. 

\begin{corollary}[$p=1$]\label{cor:P1}
With the notation analogous to Corollary \ref{cor:P2}, there exists $N(\delta)$ such that for all $n \geq N(\delta)$ we have: for any $0<\epsilon \leq 1$,
\begin{equation}
        \prob \left[ \abs{ \tfrac{n}{m}\norm{\spinv(\mA)}_F^2 -  (\alpha^*)^2} \geq \epsilon(\alpha^*)^2 \right] \leq \     \frac{n}{C_{1}\epsilon}\e^{-C_{2}n\epsilon^{4}},
    \end{equation}
    where the constants $C_{1},C_{2}>0$ may depend on $\delta$ but not on $n$ or $\epsilon$, and $\alpha^{*}=\alpha^{*}_{1}(\delta;n)$ with
    \begin{equation}
    \label{eq:limiting-astar-spinv}
    \lim_{n \to \infty} \alpha_{1}^{*}(\delta;n)
    = \sqrt{\bigg( \sqrt{\tfrac{2}{\pi}} \e^{-\tfrac{(t_1^*)^2}{2}} t^{*}_{1} - \delta(t_1^*)^2 \bigg)^{-1} -\tfrac{1}{\delta}}
    \end{equation}
    and $t^*_{1} = \sqrt{2} \cdot \erfc^{-1}(\delta)$.
\end{corollary}

\begin{proof}[Proof of Corollaries~\ref{cor:P2}--\ref{cor:P1}]
For $p \in \{1,2\}$, using Lemmata~\ref{le:boundp2}--\ref{le:boundp1}, we lower bound $-t^{*} D'_{p}(t^{*}) \geq \gamma(\delta)$ for all $n$ above some $N(\delta)$, and control $\lim_{n \to \infty} \alpha_{p}^{*}(\delta;n)$.
We conclude using Theorem~\ref{thm:frob_of_l1}.
\end{proof}

\subsection{Some remarks on the corollaries} 

Results of Corollaries \ref{cor:P2} and \ref{cor:P1} are illustrated\footnote{For reproducible research, code is available online at \url{https://github.com/doksa/altginv}.} in Figures \ref{fig:frobinv1} and \ref{fig:frobinv2}. Figure \ref{fig:frobinv1} shows the shape of the limiting $(\alpha_{p}^*)^{2}$ as a function of $\delta$, as well as empirical averages for different values of $n$ and $\delta$. As expected, the limiting values get closer to the empirical result as $n$ grows larger. In Figure \ref{fig:frobinv2} we also show the individual realizations for different combinations of $n$ and $\delta$. As predicted by the two corollaries, the variance of both $\tfrac{n}{m}\norm{\mA^\dag}_F^2$ and $\tfrac{n}{m}\norm{\spinv(\mA)}_F^2$ reduces with $n$. For larger values of $n$, all realizations are close to the limiting $(\alpha_{p}^*)^{2}$.

\begin{figure}[h!]
\centering
\includegraphics[width=.9\linewidth]{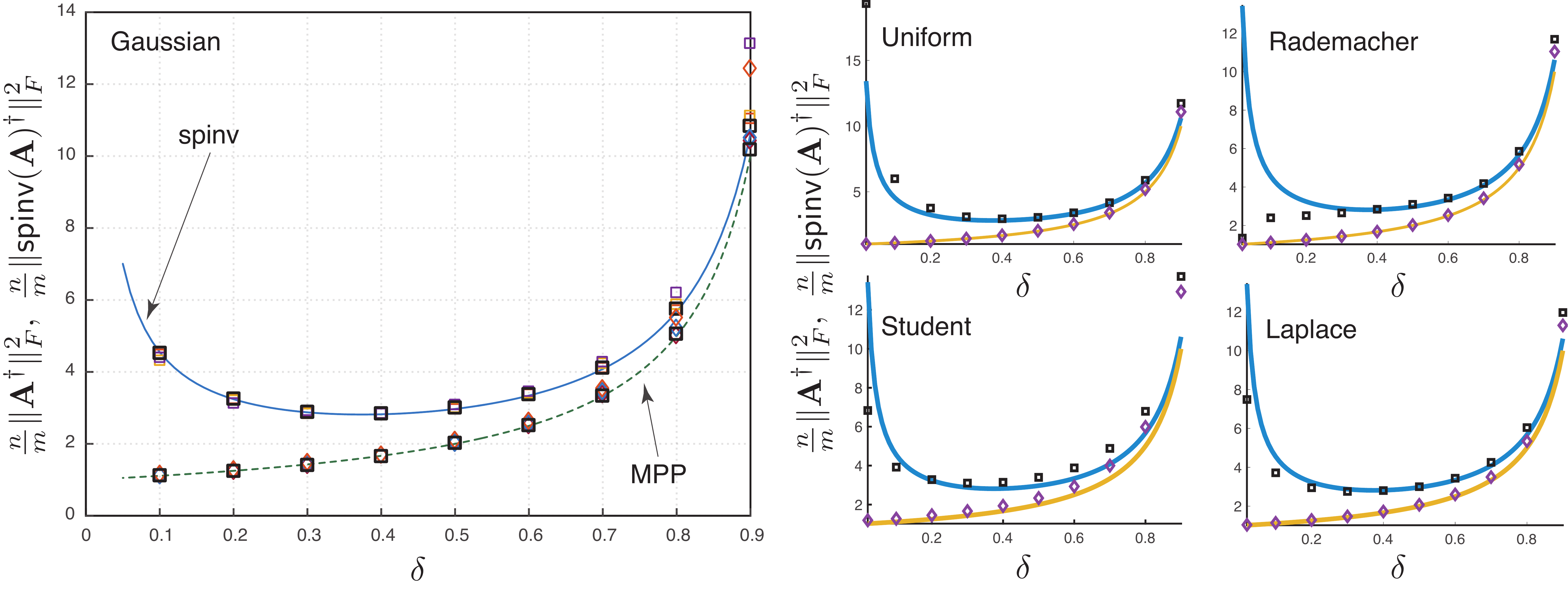}
\caption{Left: Comparison of the limiting $(\alpha_{p}^*)^2$ with the mean of $\tfrac{n}{m}\norm{\mA^\dag}_F^2$ and $\tfrac{n}{m}\norm{\spinv(\mA)}_F^2$ for 10 realizations of $\mA$. Empirical results are given for $\delta \in \set{0.1, 0.2, \ldots, 0.9}$ and $n \in \set{100, 200, 500, 1000}$. Black squares represent the empirical result for $n=1000$; colored squares represent the empirical mean of $\tfrac{n}{m}\norm{\mA^\dag}_F^2$ for $n \in \set{100, 200, 500}$, with the largest discrepancy (purple squares) for $n=100$; colored diamonds represent the empirical mean of $\tfrac{n}{m}\norm{\spinv(\mA)}_F^2$ with the largest discrepancy (orange diamonds) again for $n=100$. Right: Empirical averages over 100 trials for different matrix ensembles normalized to unit entry variance; $n = 200$ and $\delta \in \set{0.02, 0.1, 0.2, \ldots, 0.9}$.}
\label{fig:frobinv1}
\end{figure}

\begin{figure}[h!]
\centering
\includegraphics[width=0.9\linewidth]{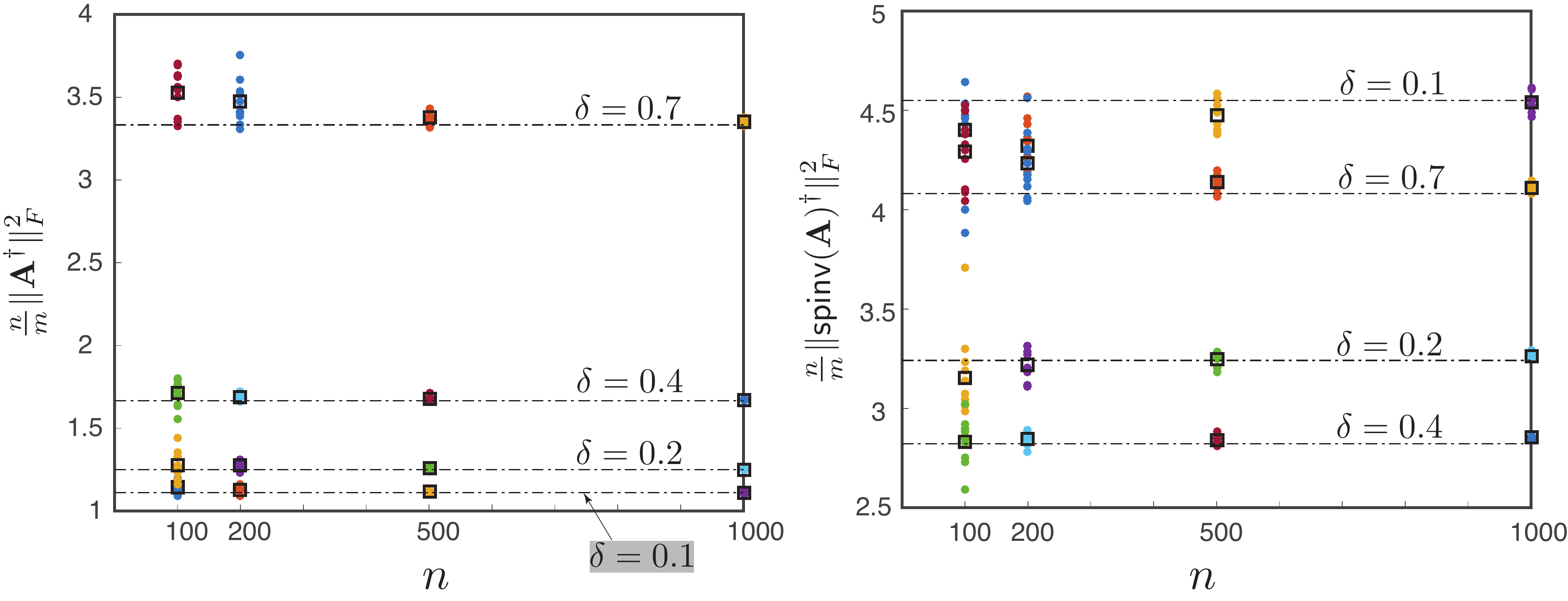}
\caption{Comparison of the limiting $(\alpha_{p}^*)^2$ with the value of $\tfrac{n}{m}\norm{\mA^\dag}_F^2$ (left) and $\tfrac{n}{m}\norm{\spinv(\mA)}_F^2$ (right) for 10 realizations of $\mA$. Results are shown for $\delta \in \set{0.1, 0.2, 0.4, 0.7}$ and different values of $n$. Values for individual realizations are shown with colored dots with different color for every combination of $n$ and $\delta$. Horizontal dashed lines indicate the limiting value for the considered values of $\delta$.}
\label{fig:frobinv2}
\end{figure}

The $\spinv$ and the MPP exhibit rather different qualitative behaviors. The Frobenius norm of the MPP monotonically decreases as $\delta$ gets smaller, while that of the $\spinv$ turns up below some critical $\delta$. Intuitively, for small $\delta$, the support of the $\spinv$ is concentrated on few entries which have to be comparably large to produce the diagonal in $\mI = \mA \mX$. A careful analysis of \eqref{eq:limiting-astar-spinv} using an asymptotic expansion of $\erfc(t)$ shows that for a sufficiently large $n$, $\alpha_1^*(\delta; n)$ behaves as $\left[\delta \log(1 / \delta)\right]^{-1/2}$ when $\delta$ is small.

The bound \eqref{eq:stability_main_bound} and the bounds in Corollaries \ref{cor:P1} and \ref{cor:P2} involve $\epsilon^4$ instead of the more common $\epsilon^2$ one gets for, e.g., Lipschitz functions: to guarantee a given probability in the right hand side of \eqref{eq:stability_main_bound}, $\epsilon$ should be of the order $n^{-1/4}$ instead of the usual $n^{-1/2}$, suggesting a comparably higher variance of $\tfrac{n}{m}\norm{\ginv{p}{\mA}}_F^2$. This is a consequence of the technique used to bound $\inf_{\abs{\alpha-\alpha^{*}} \geq \epsilon} \kappa(\alpha)-\kappa(\alpha^{*})$ in Lemma \ref{le:ArgMinKappa} which relies on strong convexity of $\kappa$. Whether a more refined analysis could lead to better error bars remains an open question.
    
Corollaries \ref{cor:P2} and \ref{cor:P1} prove that $\norm{\mA^{\dag}}_F^2$ and $\norm{\spinv(\mA)}_F^2$ indeed concentrate and give a closed-form limiting value of the optimal $\alpha^*$. It would seem natural that an interpolation to $p \in (1, 2)$ is possible, although $\alpha^*_{p}$ would be specified implicitly and computed numerically. It is less clear whether an extension to $p > 2$ is possible.

\subsection{A note on Gaussianity} Gaussian random matrices may appear as a serious restriction. It is known, however, that Gaussian matrices are a representative of a large class of random matrix models for which many relevant functionals are universal---they concentrate around the same value for a given matrix size \cite{Donoho4273,OymakTropp2015}.\footnote{We remark that universality of Gaussians is a general phenomenon that goes beyond examples in \cite{Donoho4273,OymakTropp2015}.} Although it is tempting to justify our model choice by universality, in the case of sparse pseudoinverses we must proceed with care. As Oymak and Tropp point out \cite{OymakTropp2015}, ``... geometric functionals involving non-Euclidean norms need not exhibit universality.'' They give an example of $\ell^1$ restricted minimum singular value. Indeed, as Figure \ref{fig:frobinv1} shows, while the predictions of Corollary \ref{cor:P2} for the MPP remain true over a number of ensembles and thus seem to be universal, those of Corollary \ref{cor:P1} for the $\spinv$ exhibit various levels of disagreement. For the Rademacher ensemble they collapse completely. In view of our results on pseudoinverses of structured matrices \cite{beyondMP-partI}, this comes as no surprise since we know that in this case the $\spinv$ contains the MPP. Still, from Figure \ref{fig:frobinv1}, results for Gaussian matrices are a good qualitative template for many absolutely continuous\footnote{With respect to the Lebesgue measure.} distributions, and the Gaussian assumption enables us to use sharp tools not available in a more general setting.

\subsection{A note on motivation} 

This work was originally motivated by a practical inverse problem in modern touchscreen technologies \cite{piot2015optical}. The idea is as follows: an array of light-emitting diodes (LEDs) injects light into a glass panel. Disturbances in the light propagation caused by fingers are detected by photo diodes placed next to the LEDs. With many source--detector pairs, reconstructing the multiple touch locations  becomes a tomographic problem. In order to meet the industry standards, the device must operate at a high refresh rate, yet it uses resource-constrained hardware. An obvious choice to solve the resulting over-constrained (due to coarse target resolution) tomographic problem is to apply the MPP of the forward matrix, but the refresh rate requirement precludes multiplication by a full matrix. This makes a sparse pseudoinverse attractive so long as it is stable, which translates to a controlled Frobenius norm. Even though the tomographic forward matrix is far from iid Gaussian (for example, it is sparse), it is interesting to compare it to the theoretical Gaussian results. As a toy model, we use a $15 \times 15$ pixel panel and randomly subsample the forward discrete Radon transform matrix to get a tall matrix of size $(225/\delta) \times 225$. We compute the ratio of the squared Frobenius norm of the $\mathsf{spinv}$ and the MPP and compare it to the ratio predicted by Corollaries \ref{cor:P2} and \ref{cor:P1} for Gaussian matrices. Results averaged over 10 realizations are shown in Table \ref{table:tomo}. While the two ratios are different, the Gaussian theory gives a good qualitative (and coarse quantitative) prediction.

\begin{table}[h!]
\caption{Ratio of squared Frobenius norms.}
\centering
\begin{tabular}{@{}lllllllll@{}}
\toprule
$\delta$ & 0.2 & 0.3 & 0.4 & 0.5 & 0.6 & 0.7 & 0.8 & 0.9 \\\midrule
Tomo & 3.56 & 2.73 & 2.26 & 1.98 & 1.74 & 1.55 & 1.45 & 1.32 \\
Gauss & 2.59 & 2.01 & 1.69 & 1.49 & 1.34 & 1.22 & 1.13 & 1.06 \\
\bottomrule
\end{tabular}
\label{table:tomo}
\end{table}

In general, we expect our results to be relevant whenever an application calls for multiplications by a precomputed, sparse pseudoinverse which is at the same time stable.

\section{Proof of the Main Concentration Result, Theorem \ref{thm:frob_of_l1}}
\label{sec:proof-main-concentration}

Our proof technique relies on decoupling the optimization for $\ginv{p}{\mA}$ over columns. A ``standard'' application of the CGMT would give an asymptotic result for the $\ell^p$ norm of one column which holds in probability. However, because the squared Frobenius norm is a sum of $m$ squared $\ell^2$ column norms, convergence in probability is not enough. To address this shortcoming we developed a number of technical results located mostly in Appendix \ref{appendix:stability} that lead to a stronger exponential concentration result which may be of independent interest to the users of the CGMT.

By definition, we have
\begin{equation*}
    \ginv{p}{\mA} = \argmin_{\mA \mX = \mI} \ \norm{\mX}_p = \argmin_{\mA \mX = \mI} \ \norm{\mX}_p^p,
\end{equation*}
where we assume that the solution is unique. This is true by strict convexity for $p > 1$; for $p = 1$ it holds almost surely by Theorem \ref{th:spinvmsparse}. This optimization decouples over columns of $\mX$: denoting $\mX^* = \ginv{p}{\mA}$ we have for the $i$th column that $\vx^*_i = \argmin_{\mA \vx = \ve_i} \norm{\vx}_p$. We can thus apply the following lemma proved in Section \ref{sub:proof_of_concentration}:
\begin{lemma}%
    \label{lem:concentration_column}%
    With notations and assumptions as in Theorem \ref{thm:frob_of_l1}, we have for $0<\epsilon' \leq 1$ and $n \geq \max(2/(1-\delta),N(\delta))$
    \[\prob
    \left\{ \abs{ \sqrt{n} \norm{\vx^*}_2 - \alpha^*} \geq \epsilon' \alpha^{*} \right\} \leq \
    \frac{1}{K_{1}\epsilon'} \e^{-K_{2} n \epsilon'^{4}},
    \]
    where $K_{1}, K_{2}>0$ may depend on $\delta$ but not on $n$ or $\epsilon'$.
\end{lemma}
Lemma \ref{lem:concentration_column} tells us that $\sqrt{n} \norm{\vx^*_{i}}_2$ remains close to $\alpha^{*}$ with high probability. To exploit the additivity of the \emph{squared} Frobenius norms over columns we write $(n \norm{\vx^*}_2^2 - (\alpha^*)^2) = (\sqrt{n} \norm{\vx^*}_2 - \alpha^*)(\sqrt{n} \norm{\vx^*}_2 + \alpha^*)$. Then for any $b > 0$ we have that
\begin{align*}
    \prob \left\{ \abs{ n \norm{\vx^*}_2^2 - (\alpha^*)^2} \geq \epsilon (\alpha^*)^2 \right\}
    &  \leq \prob \left\{ \abs{ \sqrt{n} \norm{\vx^*}_2 - (\alpha^*)} \geq \epsilon (\alpha^*)^2 / b \right\} \\
    & \quad +  \prob \left\{ (\sqrt{n} \norm{\vx^*}_2 + \alpha^*) \geq b \right\}.
\end{align*}
By taking $b = 3 \alpha^*$, we bound both terms using Lemma \ref{lem:concentration_column}
to obtain
\[
    \prob \left\{ \abs{ n \norm{\vx^*}_2^2 - (\alpha^*)^2} \geq \epsilon (\alpha^*)^2 \right\}
    \leq
    \frac{3}{K_{1}\epsilon} \e^{-K_{2} n (\epsilon / 3)^4} + \frac{1}{K_{1}} \e^{-K_{2} n }
    \leq
    \frac{1}{C_{1}\epsilon} \e^{-C_2 n \epsilon^4}
\]
with an appropriate choice of $C_{1},C_{2}$.

This characterizes the squared $\ell^2$ norm of one column of the MPP. The (scaled) squared Frobenius norm of $\mX^*$ is a sum of $m$ such terms which are not independent,
\(
    \frac{n}{m} \norm{\mX^*}_F^2 = \frac{n}{m} \sum_{i=1}^m \norm{\vx^*_i}_2^2,
\)
and we want to show that it stays close to $(\alpha^*)^{2}$. We work as follows: 

\begin{eqnarray*}
    \prob\left\{\abs{\tfrac{n}{m}  \norm{\mX^*}_F^2 - (\alpha^*)^2} \geq \epsilon(\alpha^*)^2\right\}
    &=& \prob\left\{\abs{\sum_{i=1}^m n\norm{\vx_i^*}_2^2 - (\alpha^*)^2} \geq m \epsilon(\alpha^*)^2\right\}\\ 
    &\leq& \prob\left\{\sum_{i=1}^m \abs{n \norm{\vx_i^*}_2^2 - (\alpha^*)^2} \geq m \epsilon(\alpha^*)^2\right\} = (*).
\end{eqnarray*}
If the sum of $m$ terms is to be larger than $m \epsilon$, then at least one term must be larger than $\epsilon$,
\begin{equation*}
\begin{aligned}
    (*) \leq \prob \left\{ \, \exists \, i, 1 \leq i \leq m: \abs{n \norm{\vx_i^*}_2^2 - (\alpha^*)^2} \geq \epsilon(\alpha^*)^2 \right\}
    &= \prob \left\{ \bigcup_{i=1}^m \set{\abs{n \norm{\vx_i^*}_2^2 - (\alpha^*)^2} \geq \epsilon(\alpha^*)^2} \right\}\\
    &\leq \sum_{i=1}^m \prob\left\{ \abs{n \norm{\vx_i^*}_2^2 - (\alpha^*)^2} \geq \epsilon(\alpha^*)^2 \right\}\\
    & \leq m\frac{1}{C_{1}\epsilon} \e^{-C_{2}n\epsilon^{4}}
    \leq \frac{\delta n}{C_{1}\epsilon} \e^{-C_{2}n\epsilon^{4}},
\end{aligned}
\end{equation*}
which is exactly the statement of Theorem \ref{thm:frob_of_l1}.

\subsection{Proof of the Main Vector Result, Lemma \ref{lem:concentration_column}} 
\label{sub:proof_of_concentration}
Define
\begin{align}
    \label{eq:no_lasso}
    \vx^{*} \ & = \ \arg\min_{\mA \vx = \ve_1} \norm{\vx}_p, \\
    \label{eq:lasso}
    \wt{\vx} \ & = \ \arg\min_{\vx} \ \norm{\mA \vx - \ve_1}_2 + \lambda \norm{\vx}_p.
\end{align}
By Lemma \ref{lem:lasso_satisfies_equality}, since the $\ell^{p}$ norm is $L$-Lipschitz with respect to the $\ell^{2}$ norm (with $L \bydef n^{\max(1/p-1/2,0)}$), if we choose $\lambda \leq \frac{\sqrt{n} - \sqrt{m}}{L} (1 - \epsilon)$, minimizers $\vx^{*}$ and $\wt{\vx}$ of \eqref{eq:no_lasso} and \eqref{eq:lasso} coincide\footnote{An analogous result does not hold for the squared lasso (except for $\lambda = 0^+$).}  with probability at least $1 - \e^{-\epsilon^2 (\sqrt{n} - \sqrt{m})^2/2}$. Using $\norm{\vy}_2 = \max_{\vu:\norm{\vu}_2 \leq 1} \vu^\T \vy$ we get
\begin{equation}
 \label{eq:frob_variational}
 \wt{\vx} = \arg\min_{\vx} \max_{\vu:\norm{\vu}_2 \leq 1} \vu^\T \mA \vx + \lambda \norm{\vx}_p - \vu^\T \ve_1.
\end{equation}
The objective in \eqref{eq:frob_variational} is a sum of a bilinear term involving $\mA$ and a convex--concave function\footnote{Convex in the first argument, concave in the second one.} $\psi(\vx, \vu) = \lambda \norm{\vx}_p - \vu^\T \ve_{1}$ as required by the CGMT (Appendix \ref{append:CGMT}). The CGMT requires $\vx$ to belong to a compact set so instead of \eqref{eq:frob_variational} we analyze the following bounded modification:
\begin{equation}
 \label{eq:po_bounded}
 \wt{\vx}_{K} = \arg\min_{\vx : \norm{\vx}_2 \leq K} \max_{\vu:\norm{\vu}_2 \leq 1} \vu^\T \mA
 \vx + \lambda \norm{\vx}_p - \vu^\T \ve_1.
\end{equation}
We will see that $\wt{\vx}_{K} = \wt{\vx}$ with high probability as soon as $K$ is large enough.

Part \ref{thmpart:cgmt_optval_concentrate} of the CGMT says that the random optimal value of the \emph{principal} optimization \eqref{eq:po_bounded} concentrates if the random value of the following \emph{auxiliary} optimization concentrates:
\begin{equation}
 \label{eq:auxiliary}
 \hat{\vx}_{K} = \arg\min_{\vx : \norm{\vx}_2 \leq K} \max_{\vu:\norm{\vu}_2 \leq 1} \norm{\vx}_2 \vg^\T \vu - \norm{\vu}_2 \vh^\T \vx + \lambda \norm{\vx}_p - \vu^\T \ve_1,
\end{equation}
where $\vg \sim \mathcal(\vec{0}, \mI_m)$ and $\vh \sim \mathcal(\vec{0}, \mI_n)$ are independent. This lets us prove that if the norm $\norm{\hat{\vx}_{K}}_2$ of the optimizer of \eqref{eq:auxiliary} concentrates, then the norm $\norm{\wt{\vx}_{K}}_2$ of the optimizer of \eqref{eq:po_bounded} concentrates around the same value.\footnote{The CGMT contains a similar statement, albeit we need a different derivation to get exponential concentration.}

We now go through a series of steps to simplify \eqref{eq:auxiliary}. For convenience we let $\vz$ be a scaled version of $\vx$, $\vz = \vx \sqrt{n}$ (accordingly $A = K \sqrt{n}$). Using the variational characterization of the $\ell^p$ norm we rewrite \eqref{eq:auxiliary} as
\begin{equation}
 \label{eq:aux_dual}
\hat{\vz}_{A} =  \arg \min_{\vz : \norm{\vz}_2 \leq A} \max_{\substack{\beta: 0 \leq \beta \leq 1\\ \vu : \norm{\vu}_2 = \beta\\\vw : \norm{\vw}_{p^*} \leq 1}} {1}/{\sqrt{n}} \cdot \norm{\vz}_2 \vg^\T \vu - {1}/{\sqrt{n}} \cdot \norm{\vu}_2 \vh^\T \vz - \vu^\T \ve_1 + {\lambda}/{\sqrt{n}} \cdot \vw^\T \vz.
\end{equation}
Let $\check{\vz}_{A}$ be a (random) value of $\vz$ at the optimum of the following reordered optimization:
\begin{equation}
    \label{eq:aux_dual_swap}
    \max_{\substack{\beta : 0 \leq \beta \leq 1\\\vw:\norm{\vw}_{p^*} \leq 1}} \min_{\vz : \norm{\vz}_2 \leq A}  \max_{\vu:\norm{\vu}_2 = \beta} {1}/{\sqrt{n}} \cdot \norm{\vz}_2 \vg^\T \vu - {1}/{\sqrt{n}} \cdot \norm{\vu}_2 \vh^\T \vz 
- \vu^\T \ve_1 + {\lambda}/{\sqrt{n}} \cdot \vw^\T \vz.
\end{equation}
By Lemma \ref{lem:phi_eps_separation} and Lemma \ref{lem:minmaxswap}, $\norm{\check{\vz}_{A}}_2$ and $\norm{\hat{\vz}_{A}}_2$ stay close with high probability (this will be made precise below). We simplify \eqref{eq:aux_dual_swap} further as follows:

\begin{align}
\label{eq:aux_dual_swap_first_line}
\eqref{eq:aux_dual_swap} 
=
&\max_{\substack{\beta : 0 \leq \beta \leq 1\\\vw : \norm{\vw}_{p^*} \leq 1}} \min_{\vz : \norm{\vz}_2 \leq A}  \beta \norm{{\norm{\vz}_2 \vg}/{\sqrt{n}} - \ve_1}_2 - {\beta}/{\sqrt{n}} \cdot \vh^\T \vz + {\lambda}/{\sqrt{n}} \cdot \vw^\T \vz \\
\nonumber =
&\max_{\substack{\beta : 0 \leq \beta \leq 1\\\vw : \norm{\vw}_{p^*} \leq 1}} \min_{\alpha : 0 \leq \alpha \leq A} \min_{\vz:\norm{\vz}_2 = \alpha}  \beta \norm{{\norm{\vz}_2 \vg}/{\sqrt{n}} - \ve_1}_2 - {\beta}/{\sqrt{n}}\cdot \vh^\T \vz + {\lambda}/{\sqrt{n}}\cdot \vw^\T \vz \\
\nonumber =
&\max_{\substack{\beta : 0 \leq \beta \leq 1\\\vw : \norm{\vw}_{p^*} \leq 1}} \min_{\alpha : 0 \leq \alpha \leq A} \beta \norm{{\alpha \vg}/{\sqrt{n}} - \ve_1}_2 - {\alpha}/{\sqrt{n}} \cdot \norm{\beta \vh - \lambda \vw }_2.
\end{align}

The objective in the last line is convex in $\alpha$ and jointly concave in $(\beta, \vw)$, and the constraint sets are all convex and bounded, so by \cite[Corollary 3.3]{Sion:1958jm} we can again exchange $\min$ and $\max$:

\begin{align}
\nonumber = &\min_{\alpha : 0 \leq \alpha \leq A} \max_{\substack{\beta : 0 \leq \beta \leq 1\\\vw : \norm{\vw}_{p^*} \leq 1}}  \beta \norm{{\alpha \vg}/{\sqrt{n}} - \ve_1}_2 - {\alpha}/{\sqrt{n}} \cdot \norm{\beta \vh - \lambda \vw }_2 \\
\nonumber = &\min_{\alpha : 0 \leq \alpha \leq A} \max_{\beta : 0 \leq \beta \leq 1} \bigg( \beta \norm{{\alpha \vg} / {\sqrt{n}} - \ve_1}_2 -  {\alpha} / {\sqrt{n}} \cdot \min_{\vw : \norm{\vw}_{p^*} \leq 1} \norm{\beta \vh - \lambda \vw }_2 \bigg) \\
\label{eq:opt_two_scalars}
= &\min_{\alpha : 0 \leq \alpha \leq A} \max_{\beta : 0 \leq \beta \leq 1} \phi(\alpha,\beta;\vg,\vh),
\end{align}
where
\begin{equation}\label{eq:DefPhi2Scalars}
\phi(\alpha,\beta;\vg,\vh) \bydef  \beta \norm{{\alpha \vg}/{\sqrt{n}} - \ve_1}_2 - {\alpha}/{\sqrt{n}} \cdot \dist(\beta\vh, \norm{\,\cdot\,}_{p^*} \leq \lambda).
\end{equation}
We thus simplified a high-dimensional vector optimization~\eqref{eq:lasso} into an optimization over two scalars~\eqref{eq:opt_two_scalars}, one of these scalars, $\alpha$, almost giving us what we seek---the (scaled) $\ell^2$ norm of $\vx$. To put the pieces together, there now remains to formally prove that the min-max switches and the concentration results we mentioned actually hold.

\subsection{Combining the ingredients}

By Lemma~\ref{lem:concentration2}, $\phi(\alpha,\beta; \vg,\vh)$ concentrates around some deterministic function $\kappa(\alpha,\beta) = \kappa_{p}(\alpha,\beta;n,\delta,\lambda)$ with $\delta = (m-1)/n$. By Lemma~\ref{le:ArgMinKappa}--3, 
\begin{equation}
\argmin_{\alpha : 0 \leq \alpha \leq A} \max_{\beta: 0 \leq \beta \leq 1} \kappa(\alpha,\beta) 
= \alpha^{*} = \alpha^{*}_{p}(\delta;n) \bydef \sqrt{\frac{D_{p}(t^{*}_{p};n)}{\delta(\delta-D_{p}(t^{*}_{p};n))}}
\end{equation}
as soon as $A=K\sqrt{n}>\alpha^{*}$ and $\lambda \leq t^{*}$, with $t^{*}=t^{*}_{p}(\delta,n)$ defined in Lemma~\ref{lem:Dp_det_prop}--Property 7. By Lemma~\ref{le:empiricalminmax}, for $0<\epsilon\leq \max(\alpha^{*},A-\alpha^{*})$, the minimizer 
\[
\alpha^{*}_{\phi} \bydef \argmin_{0 \leq \alpha \leq A} \max_{0 \leq \beta \leq 1}\phi(\alpha, \beta; \vg, \vh)
\]
stays $\epsilon$-close to $\alpha^*$ with high probability:
\[
 \prob \{ |\alpha^{*}_{\phi}-\alpha^{*}| \geq \epsilon \} \leq \zeta\big(n,\tfrac{\omega(\epsilon)}{2}\big),
\]
with 
\(
\zeta(n,\xi) = \zeta(n,\xi;A,\delta) \bydef \tfrac{c_{1}}{\xi} \e^{-c_{2} \xi^{2}n \cdot c(\delta,A)},
\)
\(
   c(\delta,A) \bydef \tfrac{\min(\delta,A^{-2})}{1+\delta A^{2}},
\)
$c_{1},c_{2}$ universal constants, and
\[
\omega(\epsilon) = \omega_{p}(\epsilon;n,\delta,\lambda) \bydef \frac{\epsilon^2}{2} \frac{\lambda\delta/t^{*}}{(1 + \delta(\alpha^* + \epsilon)^2)^{3/2}}.
\]
As a consequence, by Lemma~\ref{lem:minmaxswap}, the scaled norm $\sqrt{n} \norm{ \wt{\vx}_{K}}_2$ of the minimizer of the (bounded) principal optimization problem \eqref{eq:po_bounded} stays close to $\alpha^*$ with high probability,
\[
        \prob \big\{ \abs{\sqrt{n} \norm{\wt{\vx}_K}_2 - \alpha^*} \geq \epsilon \big\}  \leq 4 \zeta\big(n, \tfrac{\omega(\epsilon)}{2}\big).
\]
Similarly, for $0<\epsilon\leq \min(\alpha^{*},A-\alpha^{*})$ by invoking Lemma~\ref{lemma:bounded_equals_unbounded} we have that the norm $\sqrt{n} \norm{ \wt{\vx}}$ of the minimizer of the unbounded optimization \eqref{eq:frob_variational} stays close to $\alpha^*$ with high probability
\[        
    \prob \big\{ \abs{\sqrt{n} \norm{\wt{\vx}}_2 - \alpha^*} \geq \epsilon \big\}  \leq 4 \zeta\big(n, \tfrac{\omega(\epsilon)}{2}\big),
\]
and in fact \(    \prob[\wt{\vx} \neq \wt{\vx}_{K}] \leq 4\zeta\big(n,\tfrac{\omega(\epsilon)}{2}\big) \).

Since the $\ell^{p}$ norm is $L$-Lipschitz with respect to the Euclidean metric in $\R^{n}$, with $L = n^{\max(1/p-1/2,0)}$, Lemma \ref{lem:lasso_satisfies_equality} gives for any  
\(
  \lambda < \lambda_{\max}(n,m,t^{*}) \bydef
  \min \set{\tfrac{\sqrt{n} - \sqrt{m}}{2 L}, t^{*}}
\) 
that the minimizer $\vx^{*}$ of the equality-constrained optimization \eqref{eq:no_lasso} coincides with the minimizer $\wt{\vx}$ of the lasso formulation \eqref{eq:lasso}--\eqref{eq:frob_variational} except with probability at most $\e^{-n(1 - \sqrt{\delta+1/n})^2/8}$.

Overall then, for $\lambda < \lambda_{\max}$ and $\epsilon \leq \min(\alpha^{*},A-\alpha^{*})$,
\begin{equation}
\label{eq:CombiningIngredientsMainProbBound}
 \prob \big\{\abs{\sqrt{n}\norm{\vx^{*}}_2-\alpha^{*}} \geq \epsilon\big\}  \leq \e^{-n(1 - \sqrt{\delta+1/n})^2/8} + 4\zeta\big(n,\tfrac{\omega(\epsilon)}{2}\big).
\end{equation}
The infimum over admissible values of $\lambda$ is obtained by taking its value when $\lambda = \lambda_{\max}$. 

\subsection{Making the bound~\eqref{eq:CombiningIngredientsMainProbBound} explicit}

From now on we choose $A \bydef 2\alpha^{*}$ and, for $0 < \epsilon' \leq 1$, we consider $\epsilon \bydef \epsilon' \alpha^{*}$ (which satisfies $0 < \epsilon  \leq \min(\alpha^{*},A-\alpha^{*}) = \alpha^{*}$). 
We use $\lesssim_{\delta}$ and $\gtrsim_{\delta}$ to denote inequalities up to a constant that may depend on $\delta$, but not on $n$ or $\epsilon'$, provided $n \geq N(\delta)$. We specify $N(\delta)$ where appropriate.

By Lemma~\ref{le:genericboundTp}--item~\ref{it:TMax}, for any $1 \leq p \leq 2$ and any $n \geq 1$,
\begin{eqnarray*}
\lambda_{\max} = t^{*}\ \min\left(\tfrac{\sqrt{n}-\sqrt{m}}{n^{1/p-1/2}}\tfrac{1}{2t^{*}},1\right)
&\gtrsim_{\delta}& t^{*}\ \min\left(\tfrac{\sqrt{n}(1-\sqrt{m/n})}{n^{1/p-1/2}}\tfrac{1}{n^{1-1/p}},1\right)
= t^{*} \min\left(1-\sqrt{m/n},1\right).
\end{eqnarray*}
For $n \geq \tfrac{2}{1-\delta}$ we have 
\(
1-\sqrt{m/n} = 1-\sqrt{\delta+\tfrac{1}{n}} \geq 1-\sqrt{(1+\delta)/2} \gtrsim_{\delta} 1,
\)
hence $\lambda_{\max} \gtrsim_{\delta} t^{*}$.

With the shorthands $D(t) = D_{p}(t;n)$ and $D = D_{p}(t^{*};n)$, we have
\[
1+\delta (\alpha^{*}+\epsilon)^{2} \leq 1+\delta A^{2} \leq 4(1+\delta (\alpha^{*})^{2}) = 4(1+\tfrac{D}{\delta-D})= \frac{4\delta}{\delta-D} \lesssim_{\delta} (\delta-D)^{-1},
\]
hence with $\lambda = \lambda_{\max}$ we get for $n \geq 2/(1-\delta)$,
\begin{eqnarray*}
\omega(\epsilon) 
& = &
\frac{\epsilon'^{2}}{2} \frac{\lambda_{\max}}{t^{*}} \frac{\delta (\alpha^{*})^{2}}{(1+\delta(\alpha^{*}+\epsilon)^{2})^{3/2}}
\gtrsim_{\delta} 
\epsilon'^{2} \frac{ \tfrac{D}{\delta-D}}{(\tfrac{4\delta}{\delta-D})^{3/2}}
\gtrsim_{\delta}
\epsilon'^{2}  D\sqrt{\delta-D} \ .
\end{eqnarray*}
By Lemma~\ref{le:genericboundTp}--\ref{it:lowerBoundDp} we have $D = D_{p}(t^{*};n) \geq (\delta/C)^{2}$ for a universal constant $C$ independent of $n$ or $p$, hence for $n \geq 1$, $D \gtrsim_{\delta} 1$ and for $n \geq 2/(1-\delta)$,
\[
  \omega(\epsilon) \gtrsim_{\delta} \sqrt{-t^{*}D'_{p}(t^{*};n)} \epsilon'^{2}.
\]
Moreover since $\min(\delta,A^{-2}) \geq \tfrac{1}{4}\min(\delta,(\alpha^{*})^{-2}) = \tfrac{1}{4}\min(\delta,\delta(\delta-D)/D)$ we also get
\begin{eqnarray*}
c(\delta,A) 
&= &
\tfrac{\min(\delta,A^{-2})}{1+\delta A^{2}} 
\geq
\tfrac{\min(\delta, \delta(\delta-D)/D)}{\delta/(\delta-D)}
=
\min(\delta-D, \tfrac{(\delta-D)^{2}}{D}) \geq (\delta-D)^{2} = \tfrac{1}{4} [-t^{*}D'_{p}(t^{*};n)]^{2}.
\end{eqnarray*}
Since $-t^{*} D'_{p}(t^{*};n) \geq \gamma(\delta) > 0$ for any $n \geq N(\delta)$ (recall that $t^{*}= t^{*}_{p}(\delta;n)$), we have
\begin{equation}\label{eq:MainAssumption}
-t^{*}D'_{p}(t^{*}) \gtrsim_{\delta} 1,
\end{equation}
and we obtain for $n \geq \max(2/(1-\delta),N)$: $\omega(\epsilon) \gtrsim_{\delta} \epsilon'^{2}$ and $c(\delta,A) \gtrsim_{\delta} 1$. Combining the above yields, for $0<\epsilon'\leq 1$, $n \geq \max(2/(1-\delta),N(\delta))$:
\begin{equation*}
    \prob \left\{ \abs{\norm{\vx^*} - \tfrac{\alpha^*}{\sqrt{n}}} \geq \tfrac{\epsilon'\alpha^{*}}{\sqrt{n}} \right\}
    \leq    \e^{-C_{1} n}
    + 4\zeta\big(n, C_{2} \epsilon'^{2}\big)
 \leq   
 \frac{1}{K_{1}(\epsilon')^{2}} \e^{-K_{2}(\epsilon')^{4} n} 
\end{equation*}
with $K_{i},C_{i} \gtrsim_{\delta} 1$. 


\section{Conclusion} 
\label{sec:conclusion}

We studied the concentration of the Frobenius norm of $\ell^p$-minimal pseudoinverses for iid Gaussian matrices. In addition to a general result for $1 \leq p \leq 2$, we gave explicit bounds for $p \in \set{1, 2}$, that is, for the sparse pseudoinverse and the Moore-Penrose pseudoinverse. Our results show that for a large range of $m / n$ the Frobenius norm of the $\spinv$ is close to the Frobenius norm of the MPP which is the best possible among all generalized inverses (Figure \ref{fig:frobinv1}). The same does not hold for the various ad hoc strategies that yield generalized inverses with the same non-zero count (Figure \ref{fig:minor}). In applications, this means that the $\spinv$ will not blow up noise much more than the MPP. Important future directions are extensions of Theorem \ref{thm:frob_of_l1} to matrix norms other than $\ell^p$ with $1 \leq p \leq 2$, as well as matrix models other than iid Gaussian.


\section{Acknowledgments}

The authors would like to thank Mihailo Kolund\v{z}ija, Miki Elad, Jakob Lemvig, and Martin Vetterli for the discussions and input in preparing this manuscript. A special thanks goes to Christos Thrampoulidis for his help in understanding and applying the Gaussian min-max theorem, and to Simon Foucart for discussions leading to the proof of \cite[Lemma 2.1]{beyondMP-partII}.

This work was supported in part by the European Research Council, PLEASE project (ERC-StG-2011-277906).


\section*{Appendices}
\appendix

\section{Results about Gaussian processes}

\subsection{Concentration of measure} 

\begin{lemma}
\label{lem:concentration_results}
Let $\vh$ be a standard Gaussian random vector of length $n$, $\vh \sim \mathcal{N}(\vec{0}, \mI_n)$, and $f : \R^n \to \R$ a 1-Lipschitz function. Then the following hold:
\begin{enumerate}[label=(\alph*)]
    \item \label{lemitem:gausnormsingle} For any $0 < \epsilon < 1$, $\prob\{\norm{\vh}_2^2 \leq n(1 - \epsilon)\} \leq \e^{-\epsilon^2 n/ 4}$; $\prob\{\norm{\vh}_2^2 \geq n/(1 - \epsilon)\} \leq \e^{-\epsilon^2 n/ 4}$;
    \item \label{lemitem:gausnormabs} For any $0 < \epsilon < 1$, $\prob\left\{\norm{\vh}_2^2 \notin [(1 - \epsilon)n, n/(1 - \epsilon)]\right\} \leq 2\e^{-\tfrac{\epsilon^2 n}{4}}$;
    \item \label{lemitem:norm2abs} For any $\epsilon > 0$,
        $\prob\left\{\abs{\norm{\vh}_2^2 - n } \geq \sqrt{\epsilon n}\right\} \leq
        \begin{cases}
            2 \e^{-\frac{\epsilon}{8}} & \text{for $0 \leq \epsilon \leq n$}, \\
            2 \e^{-\frac{\sqrt{\epsilon n}}{8}} &\text{for $\epsilon > n$};
        \end{cases}$
    \item \label{lemitem:lipgaussingle} For any $u > 0$, $\prob\left\{f(\vh) - \E f(\vh) \geq u \right\} \leq \e^{-u^2/2}$; $\prob\left\{ f(\vh) - \E f(\vh) \leq -u\right\} \leq \e^{-u^2/2}$;
    \item \label{lemitem:lipgausabs} For any $u > 0$, $\prob\{ \abs{f(\vh) - \E f(\vh)} \geq u \} \leq 2\e^{-u^2/2}$;
    \item \label{lemitem:varoflip} $\var\{f(\vh)\} \leq 1$.
\end{enumerate}
\end{lemma}

\begin{proof}[Proofs and references]
\begin{enumerate}[label=(\alph*)]
    \item \cite[Corollary 2.3]{barvinok2005math};
    \inlineitem $\ref{lemitem:gausnormsingle}$ with a union bound;
    \inlineitem $X_i = h_i^2$ is subexponential with parameters $\nu = 2, b = 4$ \cite[Example 2.4]{wainwright}, i.e., $\E[\e^{\lambda(X_i - \mu)}] \leq \e^{\frac{\nu^2 \lambda^2}{2}}$ for all $\abs{\lambda} \leq \frac{1}{b}$, and $\mu = \E[X_i] = 1$. Applying \cite[Proposition 2.2]{wainwright}
    then yields
    \(
        \prob\left\{\abs{\norm{\vh}_2^2 - n} \geq \sqrt{\epsilon n} \right\} \leq 2
            \e^{-\frac{\epsilon}{8}},
    \)
    for $0 \leq \epsilon \leq n$, and
    \(
        \prob\left\{\abs{\norm{\vh}_2^2 - n} \geq \sqrt{\epsilon n} \right\} \leq 2
            \e^{-\frac{\sqrt{\epsilon n}}{8}},
    \)
    for $\epsilon > n$.
    \inlineitem \cite[Eq. (1.22)]{Ledoux:1999ip};
    \inlineitem union bound applied to \ref{lemitem:lipgaussingle};
    \inlineitem a consequence of Poincar\'e inequality for Gaussian measures \cite[Eq. (2.16)]{Ledoux:1999ip}: $\var\{f(\vh)\} \leq \E \{\norm{\nabla f(\vh)}_2^2\}$ for 1-Lipschitz $f$ for which $\norm{\nabla f(\vh)}_2 \leq 1$.
\end{enumerate}
\end{proof}


We will also use the following facts which can be verified by direct computation:

\begin{lemma}
\label{lem:theta_explicit}
Let $\theta(t) \bydef \E (\abs{h} - t)_+^2$, where $h \sim \mathcal{N}(0, 1)$ and $t \geq 0$. Then 
\begin{enumerate}
    \item $\theta(t) = \left(t^2+1\right) \erfc\left(\frac{t}{\sqrt{2}}\right)-\sqrt{\frac{2}{\pi }} \e^{-\frac{t^2}{2}} t $,
    \item $\theta(t) - (t/2) \theta'(t) = \erfc(t / \sqrt{2})$.
\end{enumerate}
\end{lemma}

\subsection{Convex Gaussian Min-Max Theorem (CGMT)}

\label{append:CGMT}

Let the principal optimization (PO) and auxiliary optimization (AO) be defined as 
\begin{align}
    \Phi(\mG) & \ \bydef \ \min_{\vv \in \setS_\vv} \max_{\vu \in \setS_\vu} \ \vu^\T \mG \vw + \psi(\vv,\vu),     \label{eq:po} \tag{PO}
\\
    \phi(\vg, \vh) & \ \bydef \ \min_{\vv \in \setS_\vv} \max_{\vu \in \setS_\vu} \ \norm{\vv}_2 \vg^\T \vu + \norm{\vu}_2 \vh^\T \vv + \psi(\vv, \vu),     \label{eq:ao} \tag{AO}
\end{align}
with $\mG \in \R^{m \times n}$, $\vg \in \R^m$, $\vh \in \R^n$, $\setS_\vv \subset
\R^n$, $\setS_\vu \subset \R^m$ and $\psi : \R^n \times \R^m \to \R$. 
Then we have the following result.

\begin{theorem}{\cite[Theorem 6.1]{Thrampoulidis:2016vo}}
    \label{thm:cgmt}

    In \eqref{eq:po} and \eqref{eq:ao}, let $\setS_\vv$ and $S_\vu$ be compact and $\psi$
    continuous on $\setS_\vv \times \setS_\vu$. Let also $\mG$, $\vg$, $\vh$ have
    iid standard normal entries. Then the following hold:
    \begin{enumerate}[label=(\roman*)]
        \item \label{thmpart:1} For all $c \in \R$
        \begin{equation*}
            \prob\{  \Phi(\mG) < c \} \leq 2\prob\{ \phi(\vg, \vh) \leq c \}.
        \end{equation*}
        
        \item \label{thmpart:cgmt_optval_concentrate} If $\psi(\vv,\vu)$ is additionally convex--concave on $\setS_\vv \times \setS_\vu$ where $\setS_{\vv}$ and $\setS_{\vu}$ are convex, then for all $c \in \R$
        \begin{equation*}
            \prob\{ \Phi(\mG) > c\} \leq 2 \prob\{\phi(\vg, \vh) \geq c\}.
        \end{equation*}
        In particular, for all $\mu \in \R$ and $t > 0$,
        \begin{equation*}
            \prob\{\abs{\Phi(\mG) - \mu} > t\} \leq 2\prob\{\abs{\phi(\vg, \vh) - \mu} \geq t\}.
        \end{equation*}
    \end{enumerate}
\end{theorem}


\section{Lemmata for Section \ref{sec:proof-main-concentration}}
\label{appendix:stability}

\begin{lemma}[{\cite[Lemma 9.2]{Oymak:2013vm}} with explicit dependence on $\epsilon$]
    \label{lem:lasso_satisfies_equality}
    Let $\mA \in \R^{m \times n}$ be a random matrix with iid standard normal
    entries, and $m < n$. Let further $\vy \in \R^m$ and consider the solution of an
    $\ell^2$-lasso with a regularizer $f$ which is $L$-Lipschitz with respect to the $\ell^{2}$-norm:
    \begin{equation*}
        \vx^\star \bydef \argmin_{\vx \in \R^n} \norm{\vy - \mA \vx}_2 + \lambda f(\vx).
    \end{equation*}
    Then for any $0 \leq \epsilon < 1$ and $0 < \lambda < \frac{\sqrt{n} - \sqrt{m}}{L}(1 - \epsilon)$ we have $\vy = \mA \vx^\star$ with
    probability at least $1 - \e^{-\epsilon^2 (\sqrt{n} - \sqrt{m})^2 / 2}$, that is, $\ell^2$-lasso gives the same optimizer as equality-constrained minimization.
\end{lemma}

\begin{proof}
    Using \cite[Corollary 5.35]{Vershynin:2009fv}\footnote{We actually use a one-sided variant of \cite[Corollary 5.35]{Vershynin:2009fv} which can be obtained by combining Lemma \ref{lem:concentration_results}\ref{lemitem:lipgaussingle} with the estimate of the expectation of $\sigma_{\text{min}}$, \cite[Theorem 5.32]{Vershynin:2009fv}.}, we have for any $\epsilon \geq 0$
    \begin{equation}
        \label{eq:concentration_of_sigma_min}
        \prob[\sigma_{\text{min}}(\mA^\T) / (\sqrt{n} - \sqrt{m}) \leq 1 - \epsilon] \leq \e^{-\epsilon^2(\sqrt{n} - \sqrt{m})^{2}/ 2}.
    \end{equation}
    Let $\vp \bydef \vy - \mA \vx^\star$ and $\vw \bydef \mA^\dag \vp$, where $\mA^\dag = \mA^\T (\mA \mA^\T)^{-1}$ denotes the MPP ($\mA \mA^\T$ is almost surely invertible). Since
    \(
        \norm{\vw}_2^2 = \vp^\T (\mA \mA^\T)^{-1} \vp \leq {\norm{\vp}_2^2} { \sigma_{\text{min}}^{-2}(\mA^\T)}
    \)
    we have from \eqref{eq:concentration_of_sigma_min} that for any $0 < \epsilon < 1$, with probability at least $1 -
    \e^{-\epsilon^2 (\sqrt{n} - \sqrt{m})^{2} / 2}$,
    \(
        \norm{\vw}_2 \leq \frac{\norm{\vp}_2}{(\sqrt{n} - \sqrt{m})(1 - \epsilon)}.
    \)
    Let $\vx^\circ \bydef \vx^\star + \vw$ so that $\vy - \mA \vx^\circ = \vzero$.
    Optimality of $\vx^\star$ gives
    \begin{equation}
        \label{eq:lasso_is_exact_nonpositive_diff}
        (*) = \big[ \norm{\vy - \mA \vx^\star}_2 + \lambda f(\vx^\star) \big] - \big[ \norm{\vy - \mA \vx^\circ}_2 + \lambda f(\vx^\circ) \big] \leq 0.
    \end{equation} 
    On the other hand,
    \begin{align*}
        (*)
        = \norm{\vp}_2 + \lambda f(\vx^\star) - \lambda f(\vx^\circ) 
        \geq \norm{\vp}_2 - \lambda L \norm{\vx^\star - \vx^\circ}_2 
        & = \norm{\vp}_2 - \lambda L \norm{\vw}_2 \\
        & \geq \norm{\vp}_2 \left( 1 - \tfrac{\lambda L}{(\sqrt{n} - \sqrt{m})(1 - \epsilon)} \right),
    \end{align*}
    where, by \eqref{eq:lasso_is_exact_nonpositive_diff}, the last expression must be non-positive. But if we choose
    \(
        \lambda < \frac{\sqrt{n} - \sqrt{m}}{L} (1 - \epsilon),
    \)
    the only way to make it non-positive is that $\norm{\vp}_2 = 0$.
\end{proof}

\begin{lemma}
    \label{lem:concentration2}
    Let $\vg \sim \mathcal{N}(\vzero, \mI_m)$, $\vh \sim \mathcal{N}(\vzero, \mI_n)$, $1 \leq p \leq \infty$, and define
    \begin{equation*}
    \Delta_p(\beta; \vh,\lambda) 
    \bydef   
    \tfrac{1}{\sqrt{n}}\  \dist(\beta \vh,  \norm{\,\cdot\,}_{p^*} \leq \lambda)
    \quad \text{and} \quad
    \Delta_p(\beta;n,\lambda) 
    \bydef   \E[\Delta_p(\beta;\vh,\lambda)].
     \end{equation*}
  There exist universal constants $c_1, c_2 > 0$ such that for any $0 < \epsilon < 2$, any integers $m,n$ and any $A>0$ we have,  with $\delta \bydef (m-1)/n$ and $\phi(\alpha,\beta;\vg,\vh)$ defined as in~\eqref{eq:DefPhi2Scalars}:
    \begin{equation*}
        \prob \big\{ \exists 0<\alpha \leq A, 0<\beta<1,\ \abs{\phi(\alpha, \beta ; \vg, \vh) - \kappa(\alpha, \beta)} \geq \epsilon \big\} \leq \tfrac{c_1}{\epsilon} \e^{-c_2 \epsilon^2 n \cdot \tfrac{\min(\delta,1/A^{2})}{1+\delta A^{2}}} \bydef \zeta(n, \epsilon; A,\delta),
    \end{equation*}
    where
    \begin{equation}\label{eq:DefKappa}
       \kappa(\alpha, \beta) = \kappa_{p}(\alpha,\beta; n,\delta,\lambda)
    \bydef \beta \sqrt{\delta \alpha^2 + 1} - \alpha \Delta_p(\beta;n,\lambda).
\end{equation}
\end{lemma}

\begin{proof}
    We first look at the term $\norm{\alpha \vg / \sqrt{n} - \ve_1}_2$. 
    Partitioning $\vg$ as $[g_0, \ \wt{\vg}^\T]^\T$, $\wt{\vg} \in \mathbb{R}^{m-1}$,
    \[
        \norm{\alpha \vg / \sqrt{n} - \ve_1}_2^2 = \norm{\alpha \wt{\vg} / \sqrt{n}}_2^2 + (\alpha g_0 / \sqrt{n} - 1)^2.
    \]
    Using Lemma \ref{lem:concentration_results}\ref{lemitem:gausnormabs} we get for the first term:
    \[
        \prob \left\{  \exists 0<\alpha \leq A,\  \  \norm{\alpha \wt{\vg} / \sqrt{n}}_2^2 \notin [(1 - \epsilon) \delta \alpha^2, \delta \alpha^2/(1 - \epsilon)]\right\} \leq 2 \e^{-\tfrac{\epsilon^2 (m-1)}{4}} = 2 \e^{-\tfrac{\epsilon^{2}n\delta}{4}},
    \]
    since the event $\{ \norm{\alpha \wt{\vg} / \sqrt{n}} \notin [ (1 - \epsilon) \delta \alpha^2, \delta \alpha^2 / (1 - \epsilon) ]\}$ does not depend on $\alpha > 0$.

    Next, we show that the term $(\alpha g_0 / \sqrt{n} - 1)^2$ cannot deviate much from 1: setting $\epsilon' = \epsilon/2$,
    we have $\sqrt{1-\epsilon} < 1-\epsilon'$ and $1+\epsilon' < 1/\sqrt{1-\epsilon}$ hence
    \begin{align*}
        \prob \left\{ \forall 0<\alpha \leq A,\  \left(\tfrac{\alpha g_0}{\sqrt{n}} - 1\right)^2 \in [1 - \epsilon, 1/(1 - \epsilon)]\right\} 
        \geq \
        &\prob \left\{\forall 0<\alpha \leq A,\  \abs{\tfrac{\alpha g_0}{\sqrt{n}} - 1} \in [1 - \epsilon', 1 + \epsilon']\right\} \\
        \geq \ 
        &
        \prob \left\{\forall 0<\alpha \leq A,\ \alpha g_0 / \sqrt{n} \in [- \epsilon',\epsilon']\right\}\\
        = \ &  \prob \left\{A g_0 / \sqrt{n} \ \in \ [- \epsilon',\epsilon']\right\}.
    \end{align*}
   Then, with $\erfc$ being the complementary error function and using $\erfc(z) \leq \exp(-z^{2})$ \cite{Chiani:2003eu},
   \begin{align*}
        \prob \left\{ \exists 0<\alpha \leq A,\ \ \left(\tfrac{\alpha g_0}{\sqrt{n}} - 1\right)^2 \notin [1 - \epsilon, 1/(1 - \epsilon)]\right\}
 &       \leq 
        \prob( |g_0| > \tfrac{\epsilon'\sqrt{n}}{A}) =
    \erfc \left[ \tfrac{\epsilon' \sqrt{n}}{\sqrt{2} A} \right] \\
&        \leq \e^{-\tfrac{\epsilon'^2 n}{2 A^2}}
        = \e^{-\tfrac{\epsilon^2 n}{8 A^2}}.
    \end{align*}
    Combining the above, we get that: for any $0<\epsilon<1$, setting $\epsilon' = 1-(1-\epsilon)^{2} = \epsilon(2-\epsilon) \geq \epsilon$,
    \begin{align}
        &\prob \left\{  \exists 0<\alpha \leq A,\  \ \norm{\alpha \vg / \sqrt{n} - \ve_1}_2 \notin [(1-\epsilon)\sqrt{\delta \alpha^2 + 1}, \sqrt{\delta \alpha^2 + 1}/(1-\epsilon)] \right\}\notag\\
        = \ & \prob \left\{  \exists 0<\alpha \leq A,\  \ \norm{\alpha \vg / \sqrt{n} - \ve_1}_2^2 \notin [(1-\epsilon')(\delta \alpha^2 + 1), (\delta \alpha^2 + 1)/(1-\epsilon')] \right\}\notag\\
        \leq \ & 2\e^{-\tfrac{\epsilon'^2 n \delta}{4}} + \e^{-\tfrac{\epsilon'^2 n}{8A^2}} 
         \leq \ c'_1 \e^{-c'_2 \epsilon'^2 n \cdot \min(\delta,1/A^{2})} \leq c'_{1} \e^{-c'_{2}\epsilon^{2}n \cdot \min(\delta,1/A^{2})},\label{eq:concentrationlemmaterm1}
    \end{align}
    where the constants $c'_{1},c'_{2}$ are universal.  
    
    For the second term in $\phi(\alpha, \beta; \vg, \vh)$, we note that Euclidean distance to a convex set $\dist(\vx, \cal C)$ is 1-Lipschitz in $\vx$ with respect to $\norm{\,\cdot \,}_2$ so by Lemma \ref{lem:concentration_results}\ref{lemitem:lipgausabs} we get for $0<\beta \leq 1$ (we omit the dependency in $\lambda$ for brevity) and any $\tau \geq 0$ that
    \begin{equation}
        \prob \left\{ \abs{ \Delta_p(\beta; \vh) - \Delta_p(\beta)} > \tau \right\}
        \leq 2 \e^{- \tfrac{\tau^2 n}{2\beta^{2}}} \leq 2 \e^{- \tfrac{\tau^2 n}{2}}.\label{eq:BoundFh}
    \end{equation}
    This obviously extends to $\beta = 0$ since $\Delta_{p}(0;\vh) = \Delta_{p}(0) = 0$.

    Next, we want to bound 
    \[ 
        \prob \left\{ \exists \beta \in [0, 1], \ \abs{ \Delta_p(\beta; \vh) - \Delta_p(\beta)} > \tau \right\}
        =
        \prob \left\{ \sup_{\beta \in [0, 1]} \abs{ \Delta_p(\beta; \vh) - \Delta_p(\beta)} > \tau \right\}.
    \] 
    By Lemma \ref{lem:sqrtn_lipschitz}, the function $\beta \mapsto f_{\vh}(\beta) \bydef \abs{\Delta_p(\beta; \vh) - \Delta_p(\beta)}$ is $L_{\vh}$-Lipschitz in $\beta$ with $L_{\vh} \bydef \max \set{\norm{\vh}_2/\sqrt{n}, 1}$. Hence $f_{\vh}$ is continuous, and its supremum on the closed interval $[0,\ 1]$ is indeed a maximum reached at some maximizer $\beta_{\vh}$. 
    
    Let $b \leq 1/2$ and $Y_{b} = \set{b\tau/2,3b\tau/2,\ldots,(k-1/2)b\tau}$ be a uniform sampling of $[0,
    1]$ with spacing $b \tau$, with the last segment possibly being
    shorter. For a given $\vh$, there exists $y_\vh \in Y_{b}$ such that $\abs{\beta_{\vh} - y_{\vh}} \leq b \tau$. 
    For this $y_{\vh}$ we write
    \[
        f_{\vh}(\beta_{\vh}) - f_{\vh}(y_{\vh}) \leq \abs{ f_{\vh}(\beta_{\vh}) - f_{\vh}(y_{\vh}) } 
        \leq L_{\vh} \abs{\beta_{\vh} - y_{\vh}} \leq L_{\vh} b \tau,
    \]
    so that 
    \begin{eqnarray*}
          \prob\{ \sup_{\beta \in [0, 1]} f_{\vh}(\beta) > \tau\} 
          &=&      \prob\{ f_{\vh}(\beta_\vh) > \tau \} 
          \ \leq\  \prob\{ f_{\vh}(y_\vh) + L_{\vh} b\tau > \tau\}
          \ \leq\  \prob\left\{ f_{\vh}(y_\vh) > \tau/2~\text{or}~L_{\vh} b\tau > \tau/2\right\}\\
          &\leq& \prob\{ f_{\vh}(y_\vh) > \tau/2\} + \prob\{\max \set{\norm{\vh}_2/\sqrt{n},1} b > 1/2\}.
      \end{eqnarray*}
      As we do not know a priori to which $y \in Y_{b}$ the maximizer $\beta_{\vh}$ will be close, we continue with a union bound, and we further use that $b \leq 1/2$ to obtain by~\eqref{eq:BoundFh} and Lemma \ref{lem:concentration_results}\ref{lemitem:gausnormsingle}
          \begin{eqnarray*}
            \prob\{ \sup_{\beta \in [0, 1]} f_{\vh}(\beta) > \tau\}
      & \leq & \prob\Big\{ \exists y \in Y_b, f_{\vh}(y) > \tau/2\Big\} + \prob\Big\{\norm{\vh}_2^{2} > \tfrac{n}{4b^{2}}\Big\}
         \leq   |Y_{b}| \  2 \e^{-\tfrac{\tau^2 n}{8}} + \e^{-\tfrac{(1-4b^{2})^{2}n}{4}}\\
        & \leq & (1+1/(b\tau)) 2 \e^{-\tfrac{\tau^2 n}{8}} + \e^{-\tfrac{(1-4b^{2})^{2}n}{4}}.
    \end{eqnarray*}
    Setting
    $b \bydef \tfrac{1}{2}\sqrt{\tfrac{2 - \sqrt{2}}{2}} \approx 0.27$ we get that for any $0<\tau<1$ we have $\tfrac{(1-4b^{2})^2}{4} = \tfrac{1}{8} \geq \tfrac{\tau^2}{8}$, hence
    \begin{equation}
       \label{eq:concentrationlemmaterm2}
      \prob\{ \sup_{\beta \in [0, 1]} f_{\vh}(\beta) > \tau\}
      \leq
      (3+2/(b\tau)) \e^{-\tfrac{\tau^2 n}{8}}
       <  \tfrac{11}{\tau} \e^{-\tfrac{\tau^2 n}{8}}, ~ 0 < \tau < 1.
    \end{equation}

    To conclude we combine concentration bounds for both terms.
    First, we observe that
    \[
    \sup_{\substack{0 \leq \alpha \leq A\\ 0 \leq \beta \leq 1}} \abs{\phi(\alpha,\beta;\vg,\vh)-\kappa(\alpha,\beta)}
    \leq 
    \sup_{0 \leq \alpha \leq A} \abs{\norm{\frac{\alpha \vg}{\sqrt{n}}-\ve_{1}}_2-\sqrt{1+\delta\alpha^{2}}}
    +
    A \sup_{0 \leq \beta \leq 1} \abs{\Delta_{p}(\beta;\vh)-\Delta_{p}(\beta)}
    \]
    and by a union bound we just need to control the probability that each term exceeds $\epsilon/2$. Since we assume that $0<\epsilon<2$, we can use the multiplicative control~\eqref{eq:concentrationlemmaterm1} as follows:
    \begin{equation*}
    \begin{aligned}
      \prob \left\{  \exists \alpha \in (0,A], \ \Big|\norm{\tfrac{\alpha \vg}{\sqrt{n}} - \ve_1}_2 \!\! - \! \sqrt{\delta \alpha^2 + 1} \Big|  > \tfrac{\epsilon}{2} \right\}
      & =   \prob \left\{  \exists \alpha \in (0,A], \Bigg|\tfrac{\norm{\tfrac{\alpha \vg}{\sqrt{n}} - \ve_1}_2}{\sqrt{\delta \alpha^2 + 1}}-1 \Bigg| >\tfrac{\epsilon}{2\sqrt{\delta \alpha^2 + 1}} \right\}\\
      & \leq   \prob \left\{  \exists \alpha \in (0,A], \Bigg|\tfrac{\norm{\tfrac{\alpha \vg}{\sqrt{n}} - \ve_1}_2}{\sqrt{\delta \alpha^2 + 1}}-1\Bigg| >
      \tfrac{\epsilon}{2\sqrt{\delta A^2 + 1}} \right\}\\
      & \leq   \prob \left\{ \exists \alpha \in (0,A], \tfrac{\norm{\alpha \vg / \sqrt{n} - \ve_1}_2}{\sqrt{\delta \alpha^2 + 1}} \notin   [1-\epsilon',\tfrac{1}{1-\epsilon'}] \right\}\\
      &\leq c'_1 \e^{-c'_2 \epsilon'^2 n \cdot \min(\delta,1/A^{2})}.
    \end{aligned}
    \end{equation*}
    provided that $[1-\epsilon',1/(1-\epsilon')] \subset [1-\tfrac{\epsilon}{2\sqrt{\delta A^2 + 1}},1+\tfrac{\epsilon}{2\sqrt{\delta A^2 + 1}}]$. This is achieved with $\epsilon' = \epsilon/(1+\sqrt{\delta A^{2}+1})$. Combining the resulting  bound with the bound on $\prob\{ \sup_{\beta \in [0, 1]} f_{\vh}(\beta) > \tfrac{\epsilon}{2A}\}$ from~\eqref{eq:concentrationlemmaterm2} yields the result.
\end{proof}

\begin{corollary}\label{cor:empiricalmaxbeta}
    Consider $A>0$ and define for any set $S \subseteq [0, A]$:
    \begin{align*}
        \phi(\alpha ; \vg, \vh) & \bydef  \sup_{0 \leq \beta \leq 1} \phi(\alpha, \beta ; \vg, \vh), & 
        \kappa(\alpha) & \bydef  \sup_{0 \leq \beta \leq 1} \kappa(\alpha, \beta), \\
        \phi_S(\vg, \vh) & \bydef  \inf_{\alpha \in S} \phi(\alpha ; \vg, \vh), &
        \kappa_S & \bydef  \inf_{\alpha \in S} \kappa(\alpha).
    \end{align*}
    With $\vg \sim \mathcal{N}(\vzero, \mI_m)$, $\vh \sim \mathcal{N}(\vzero, \mI_n)$, $\delta = (m-1)/n$, we have for $0<\epsilon<2$ that
    \begin{equation}
        \label{eq:concentration_of_min_alpha}
        \prob \left\{ \sup_{S \subset [0,A]} \abs{\phi_S(\vg, \vh) - \kappa_S} \geq \epsilon \ \right\} \leq \zeta(n, \epsilon).
    \end{equation}
    with $\zeta(n,\epsilon)=\zeta(n,\epsilon;A,\delta)$ defined in~Lemma \ref{lem:concentration2}. In particular, letting $\phi(\alpha) \bydef \phi(\alpha ; \vg, \vh)$,
    \begin{equation}
        \label{eq:concentration_of_max_beta}
        \prob \left\{ \sup_{0 \leq \alpha \leq A}\abs{\phi(\alpha) - \kappa(\alpha)} \geq \epsilon \right\} \leq \zeta(n, \epsilon).
    \end{equation}
\end{corollary}
\begin{proof}
To lighten notation we suppress the dependence of the stochastic function $\phi$ on random vectors $\vg$ and $\vh$. By Lemma~\ref{lem:concentration2} we have with probability at least $1-\zeta(n,\epsilon)$: for all $0 \leq \alpha \leq A$ and $0 \leq \beta \leq 1$, $\abs{\phi(\alpha,\beta)-\kappa(\alpha,\beta)} \leq \epsilon$. When this holds we have for any $S \subset [0,A]$:
\begin{align}
    \phi_{S} &= \inf_{\alpha \in S} \phi(\alpha) \leq \inf_{\alpha \in S} [\kappa(\alpha) + \epsilon] = \kappa_{S}+\epsilon, \\
    \phi_{S} &= \inf_{\alpha \in S} \phi(\alpha) \geq \inf_{\alpha \in S} [\kappa(\alpha)-\epsilon] = \kappa_{S}-\epsilon.
\end{align}
\end{proof}
We will shortly characterize $\kappa(\alpha,\beta)$ and $\kappa(\alpha)$ using  properties of the following quantity:
\begin{equation}
    \label{eq:DefDp}
    D_{p}(t;n) \bydef \left(\E[    \tfrac{1}{\sqrt{n}}\  \dist(\vh,  \norm{\,\cdot\,}_{p^*} \leq t)]\right)^{2}.
\end{equation}
    
\begin{lemma}[Deterministic properties of $D_p$]
    \label{lem:Dp_det_prop}
    Let $t \geq 0$, $1 \leq p \leq \infty$, and define
    \(
      \mathcal{C}_{t} = \mathcal{C}_{t,p} \bydef \set{\vx \in \R^{n}: \norm{\vx}_{p^{*}} \leq t}
    \)
    and $D_p(t;n)$ as in~\eqref{eq:DefDp}. Using $D_{p}(t)$ as a shorthand, the following hold:
    \begin{enumerate}
        \item The sets $\mathcal{C}_{t}$ are convex and nested with $\mathcal{C}_{t} \subsetneq \mathcal{C}_{t'}$ for $t < t'$;
        \item For any vector $\vh$, the function $t \mapsto \dist(\vh,\mathcal{C}_{t})$ is non-increasing and convex;
        \item $D_{p}(t)$ is a (strictly) decreasing convex function of $t$;
        \item $\lim_{t \to \infty} D_{p}(t) = 0$;
        \item $\frac{n}{n+1} \leq D_{p}(0) \leq 1$;
        \item The function $t\mapsto D_{p}(t)$ is infinitely differentiable;
        \item Let $g(t) = g(t;n) \bydef D_p(t) -\frac{t}{2}D'_{p}(t)$. For any $0 < \delta < D_{p}(0)$ there is a unique 
        \[t^{*} = t^{*}_{p}(\delta;n) \in (0,\infty)\] such that $g(t) > \delta$ for $t < t^*$ and $g(t) < \delta$ for $t > t^*$. It holds that $D_{p}(t^{*}) < g(t^{*}) = \delta$.
    \end{enumerate}
\end{lemma}

\begin{proof}
\begin{enumerate}
  \item Obvious. 
  \inlineitem We recall that $\vx \mapsto \dist(\vx, {\cal C})$ is convex for convex $\cal C$ \cite[Example 3.16]{Boyd:2004uz}.
  Next, $(\vx, t) \mapsto t \dist(\vx / t, {\cal C})$ is convex in both
  arguments because it is the perspective of $\vx \mapsto \dist(\vx, {\cal
  C})$ \cite[Chapter 2]{Boyd:2004uz}. Applying this to ${\cal C} = {\cal C}_{1}$ and observing that $\dist(\vh,{\cal C}_{t}) = t\dist(\vh/t,{\cal C}_{1})$ we obtain that $\dist(\vh,{\cal C}_{t})$ is convex in $t$. The fact that it is non-increasing follows from Property~1.
  \item Since expectation of convex functions is convex, and the square of a non-negative convex function is convex, $D_p(t)$ is convex as claimed. That it is (strictly) decreasing is obvious.
  \item For any $\vy \in \R^{n}$ and $p\geq 1$, $\norm{\vy}_{p^{*}} \leq \norm{\vy}_{1} \leq \sqrt{n} \norm{\vy}_2$. Hence, for any given $t > 0$, $\dist(\vy, \mathcal{C}_{t}) = 0$ as soon as $\norm{\vy}_2 \leq t / \sqrt{n}$ and we can write for all $p \geq 1$ that
  \begin{align*}
    \E[  \dist(\vh, \mathcal{C}_{t}) ]
    &= \int_{\vb \in \mathbb{S}^{n-1}} \int_{r = t/\sqrt{n}}^{\infty} \dist(r\vb, \mathcal{C}_{t}) \ p_\vh(r\vb) \ \mu(\di \vb) \ r^{n-1} \, \di r \\
    &\leq \int_{\vb \in \mathbb{S}^{n-1}} \int_{r = t/\sqrt{n}}^{\infty} r \ p_\vh(r\vb) \ \mu(\di \vb) \ r^{n-1} \, \di r 
    = Z_n \mu(\mathbb{S}^{n-1}) \int_{r = t/\sqrt{n}}^{\infty} r^n \e^{-r^2 / 2} \ \di r, 
  \end{align*}
  where $Z_n = (2\pi)^{-n/2}$, $\mu(\mathbb{S}^{n-1}) = 2\pi^{n/2} / \Gamma \left( \frac{n}{2} \right)$, and $\Gamma \left( \cdot \right)$ being the gamma function. The last expression vanishes as $t \to \infty$.

  \item The upper bound follows from Property 3 and Jensen's inequality.
  To get the lower bound we compute $n D_p(0) = \left[ \E \norm{\vh}_2 \right]^2$ by integration in polar coordinates,
  \begin{align*}
      \E \{ \dist(\vh, \mathcal{C}_{0}) \} = \E \{ \norm{\vh}_2 \}
      & \ = \ \int_{\vb \in \mathbb{S}^{n-1}} \int_{r = 0}^{\infty} r \ p_\vh(r\vb) \ \mu(\di \vb) \ r^{n-1} \, \di r \\
      & \ = \ Z_n \mu(\mathbb{S}^{n-1}) \int_{r = 0}^{\infty} r^n \e^{-r^2/2} \, \di r \\
      & \ \stackrel{(a)}{=} \ (2 \pi)^{-n/2} \frac{ 2 \pi^{n / 2}}{\Gamma \left( \frac{n}{2} \right)} \int_{s = 0}^{\infty} 2^{\frac{n-1}{2}}s^{\frac{n-1}{2}} \e^{-s} \, \di r \\
      & \ \stackrel{(b)}{=} \ {\sqrt{2} \, \Gamma \left( \frac{n+1}{2} \right) } \bigg/ {\Gamma \left( \frac{n}{2} \right) },
  \end{align*}
  where in $(a)$ we used the substitution $u = r^2/2$, and in $(b)$ we invoked the definition of the gamma function, $\Gamma(z) = \int_0^\infty x^{z - 1} \e^{-x} \di x$. We now use the inequality of Wendel, \cite[Eq. (7)]{Wendel:1948fv},
  \(
      {\Gamma(x + a)} \big/ {\Gamma(x)} \geq x(x + a)^{a-1}
  \)
  to conclude that 
  \begin{equation}\label{eq:lowerBoundExpNormH}
  \E \{ \norm{\vh}_2 \} \geq n (n + 1)^{-1/2}
  \end{equation}
  and $D_p(0) \geq n / (n + 1)$.
  \item Using $\dist(\vh, \mathcal{C}_{t}) = t \dist(\vh / t, \mathcal{C}_1)$ and a change of variable $r = t \rho$ we obtain:
  \begin{align}
      q(t) \bydef \sqrt{n D_p(t)} = 
      Z_n t^{n+1} \int_{\vb \in \mathbb{S}^{n-1}} \int_{\rho = 0}^{\infty} \dist(\rho \vb, {\cal C}_1) \e^{-\rho^2 t^2 / 2}  \, \mu(\di \vb) \rho^{n-1} \, \di \rho 
      \label{eq:Dp_for_inf_diff}
  \end{align}
  so that $D_p(t)$ is infinitely differentiable (by the dominated convergence theorem). 
  \item In particular, as 
  \(
      \abs{\parder{}{t} \dist(\rho \vb, {\cal C}_1) \e^{-\rho^2 t^2 / 2} \rho^{n-1}} 
      = \dist(\rho \vb, {\cal C}_1) t \rho^{n+1}  \e^{-\rho^2 t^2 / 2}
      \leq  \rho^{n+2} t  \e^{-\rho^2 t^2 / 2}
  \)
  where the rightmost expression is integrable for every $t > 0$ and $n \in \N$, we can differentiate under the integral sign in~\eqref{eq:Dp_for_inf_diff} to get 
  \begin{align}
      -\tfrac{t}{2} D_p'(t) &= -\tfrac{t}{n} q(t) q'(t) \nonumber \\
      &= -\tfrac{1}{n} q(t) \left( (n+1) q(t) - Z_n t^{n+3} \int_{\mathbb{S}^{n-1}} \int_{0}^{\infty} \dist(\rho \vb, {\cal C}_1) \rho^{n+1} \e^{-\rho^2 t^2 / 2}  \, \mu(\di \vb) \, \di \rho \right).\label{eq:CpDevelopedExpression}
  \end{align}
  All terms can be seen to vanish as $t \to \infty$ by arguments analogous to those in the end of the proof of Property 4, hence $\lim_{t \to \infty} [-\tfrac{t}{2} D_p'(t)] = 0$.
   
  Since $D_p(t)$ is strictly decreasing we have $D_p'(t) < 0$. Since it is convex, $D_p''(t) \geq 0$. Thus
  \[
      g'(t) = D_p'(t) - \tfrac{1}{2} D_p'(t) - \tfrac{t}{2} D_p''(t) = \tfrac{1}{2} D_p'(t) - \tfrac{t}{2} D_p''(t) < 0
  \]
  for $t > 0$, meaning that $g(t)$ is strictly decreasing. Since $\lim_{t \to \infty} [-\tfrac{t}{2} D_p'(t)] = 0$ and $\lim_{t \to \infty} D_p(t) = 0$, we have $\lim_{t \to \infty} g(t) = 0$. It follows that for $0 < \delta < \lim_{t \to 0} g(t)$, there is a unique $t^*(\delta)$, such that $0 < t^*(\delta) < \infty$ and $g(t) > \delta$ for $t < t^*$ and $g(t) < \delta$ for $t > t^*$. Finally, since $g(t) \geq D_{p}(t)$ we have $\lim_{t \to 0} g(t) \geq D_{p}(0)$.
\end{enumerate}
\end{proof}

\begin{lemma}
\label{le:genericboundTp}
  Denote
  \begin{equation}\label{eq:DefTheta}
    \theta(t) \bydef \E (\abs{h}-t)_{+}^{2},
  \end{equation}
  with $h$ being a standard normal variable and $(\, \cdot \,)_{+} = \max(\, \cdot \, , \ 0)$ the positive part. This is a strictly decreasing function of $t$ with $\theta(0)=1$ and $\lim_{t \to \infty}\theta(t) = 0$. Recalling that $t^*_p(\delta;n)$ is the unique solution to $D_p(t; n) - \tfrac{t}{2} D_p'(t;n) = \delta$ on $(0, \infty)$, the following holds for all $1 \leq p \le \infty$, $0<\delta<1$:
  \begin{enumerate}
    \item\label{it:TMax} For any $n \geq 1$,
  \begin{equation}
    \label{eq:TMax}
    t^{*}_{p}(\delta;n) \leq  2\ \theta^{-1}(\delta)\  n^{1-1/p}.
  \end{equation}
  \item\label{it:TMin} For any $n \geq \tfrac{2}{1-\delta}$,
  \begin{equation}\label{eq:TMin}
    t^{*}_{p}(\delta;n) \geq \theta^{-1}(\tfrac{1+\delta}{2}) > 0.
  \end{equation}
  \item\label{it:TMin2} For any $n \geq  \max\left(\tfrac{2}{1-\delta},\tfrac{1}{\delta}\right)$,
  \begin{equation}\label{eq:TMin2}
    t^{*}_{p}(\delta;n) \geq \max\left(\theta^{-1}(2\delta),\theta^{-1}(\tfrac{1+\delta}{2})\right) > 0.
  \end{equation}
   \item\label{it:lowerBoundDp} There is a universal constant $C$ independent of $\delta$, $p$ and $n$ such that for all $n \geq 1$ we have
  \begin{equation}
    \label{eq:lowerBoundDp}  
    D_{p}(t^{*}_{p};n)  \geq \left({\delta}/{C}\right)^{2},\\
  \end{equation}
  where we use the shorthand $t^{*}_{p} = t^{*}_{p}(\delta;n)$.
  \end{enumerate}
\end{lemma}
\begin{lemma}\label{le:boundp2}
With the notations of Theorem~\ref{thm:frob_of_l1}, for $p=2$,
\begin{equation}
\lim_{n \to \infty} -t^{*}_{2} D'_{2}(t^{*}_{2};n)  =  (1-\delta)\delta > 0, \quad \quad \quad
\lim_{n \to \infty}\alpha^{*}_{2}(\delta;n)  =  \tfrac{1}{\sqrt{1-\delta}}.
\label{eq:UnivLowBoundDprime2_and_AlphaD2}
\end{equation}
\end{lemma}

\begin{lemma}\label{le:boundp1}
 With the notations of Theorem~\ref{thm:frob_of_l1} for $p=1$, denoting $t^*_{1}(\delta) \bydef \sqrt{2} \cdot \erfc^{-1}(\delta)$,
\begin{align}
\lim_{n \to \infty}  -t^{*}_{1} D'_{1}(t^{*}_{1};n)  >  0\label{eq:UnivLowBoundDprime1}, \quad \quad
\alpha^{*}_{1}(\delta) \bydef \lim_{n \to \infty} \alpha^{*}_{1}(\delta;n) 
= \sqrt{\bigg(\sqrt{\tfrac{2}{\pi}} \e^{-\tfrac{(t_1^*)^2}{2}} t^{*}_{1} - \delta(t_1^*)^2\bigg)^{-1} -\tfrac{1}{\delta}}. \nonumber
\end{align}
\end{lemma}

\begin{proof}[Proof of Lemmata~\ref{le:genericboundTp}--\ref{le:boundp2}--\ref{le:boundp1}]
~
\begin{itemize}
  \item {\bf Step 1.} We show that for all $p,n,t$ we have
  \begin{equation}
  \label{eq:D1Dp}
    D_p(t;n)  \geq D_1(t;n) \geq D_p\left(t n^{1-1/p};n\right).
  \end{equation}
  The inequalities $\norm{\, \cdot \,}_{p^{*}} \geq \norm{\, \cdot \,}_\infty 
  \geq n^{-1/p^{*}} \norm{\, \cdot \,}_{p^{*}} = n^{-(1-1/p)} \norm{\, \cdot \,}_{p^{*}}$ imply  $\mathcal{C}_{t,p} \subset \mathcal{C}_{t,1} \subset \mathcal{C}_{t n^{1-1/p},p}$. It follows that $d(\cdot,\mathcal{C}_{t,p}) \geq d(\cdot,\mathcal{C}_{t,1}) \geq d(\cdot,\mathcal{C}_{t n^{1-1/p},p})$, which yields~\eqref{eq:D1Dp}.
  \item {\bf Step 2.} We establish that for any $p,n,\delta$, with the shorthand $t^{*}_{p} = t^{*}_{p}(\delta;n)$,
  \begin{eqnarray}
    \label{eq:lowerBoundDp1}
    D_{p}(t^{*}_{p}/2;n) &\geq& \delta,\\
    \label{eq:upperBoundDp}  
    D_{p}(t^{*}_{p};n)  & \leq & \delta.
  \end{eqnarray}
       
  With another shorthand $D_{p}(t) \bydef D_{p}(t;n)$ and the convexity of $D_{p}(t)$, we have for all $t,h$
  \begin{equation}
  \label{eq:DpConvex}
  D_{p}(t+h) \geq D_{p}(t) + h D'_{p}(t).
  \end{equation}
  Applying it to $t=t^{*}_{p}$ and $h=-t^{*}_{p}/2$ and using the definition of $t^{*}_{p}$, we get 
  \[
  D_{p}(t^{*}_{p}/2) \geq D_{p}(t^{*}_{p})-(t^{*}_{p}/2)D'_{p}(t^{*}_{p}) = \delta,
  \] 
  i.e.,~\eqref{eq:lowerBoundDp1} holds.  Since $D_{p}(t)$ is non-increasing, we have $D'_{p}(t) \leq  0$ and $D_{p}(t) -\tfrac{t}{2} D_p'(t) \geq D_{p}(t)$ for any $t$. Applying to $t=t^{*}_{p}$ this establishes~\eqref{eq:upperBoundDp} by definition of $t^{*}_{p}$. 
  
  \item {\bf Step 3.} Let $\bar{D}_{p}(t;n) \bydef \tfrac{1}{n} \E~\dist^2(\vh, \norm{\,\cdot\,}_{p^*} \leq t)$. By Jensen's inequality, for any $p,n,t$:
  \begin{equation}
     D_{p}(t;n) =  \tfrac{1}{n}\big(\E~\dist(\vh, \norm{\,\cdot\,}_{p^*} \leq t)\big)^{2} \leq \tfrac{1}{n} \E~\dist^2(\vh, \norm{\,\cdot\,}_{p^*} \leq t) = \bar{D}_{p}(t;n).
     \label{eq:UpperBoundDpBarDp}
  \end{equation}
  For $p=1$, $p^* =\infty$ we compute
  \begin{align}
    \bar{D}_{1}(t) \bydef \tfrac{1}{n}\E~\dist^2(\vh, \norm{\,\cdot\,}_{\infty} \leq t) 
    = \tfrac{1}{n} \E~\norm{\vh - \proj_{\norm{\,\cdot\,}_{\infty} \leq t} \vh}_{2}^2
     = \tfrac{1}{n} \sum_{i = 1}^n \E  (\abs{h_{i}} - t)_{+}^2 = \theta(t).
    \label{eq:barD1theta}
  \end{align}

  \item {\bf Step 4.} We get \eqref{eq:TMax} by combining the previous steps. For any $p,\delta,n$:
  \begin{eqnarray*}
    \theta(t^{*}_{p} n^{1/p-1}/2) & \stackrel{\eqref{eq:barD1theta}}{=} & \bar{D}_{1}(t^{*}_{p} n^{1/p-1}/2)
     \stackrel{\eqref{eq:UpperBoundDpBarDp}}{\geq} D_{1}(t^{*}_{p} n^{1/p-1}/2) \stackrel{\eqref{eq:D1Dp}}{\geq} D_{p}(t^{*}_{p}/2) \stackrel{\eqref{eq:lowerBoundDp1}}{\geq} \delta.
  \end{eqnarray*}

  \item {\bf Step 5.}  Since the function $f: \vh \mapsto f(\vh) \bydef \dist(\vh, \norm{\,\cdot\,}_{p^{*}} \leq t)$ is $1$-Lipschitz, and $\vh$ standard normal, we can apply Lemma \ref{lem:concentration_results}\ref{lemitem:varoflip}
  to show that for any $p,n,t$,
  \begin{equation}\label{eq:LowerBoundDpBarDp}
    D_p(t;n) \geq \bar{D}_{p}(t;n) -\tfrac{1}{n}.
  \end{equation}

  \item {\bf Step 6.} Combining with the previous steps yields for any $p,\delta$, and $n>\tfrac{1}{1-\delta}$:
  \begin{eqnarray*}
  \theta(t^{*}_{p}) & \stackrel{\eqref{eq:barD1theta}}{=} & \bar{D}_{1}(t^{*}_{p})
    \stackrel{\eqref{eq:LowerBoundDpBarDp}}{\leq} D_{1}(t^{*}_{p})+\tfrac{1}{n} 
    \stackrel{\eqref{eq:D1Dp}}{\leq} D_{p}(t^{*}_{p})+\tfrac{1}{n} \stackrel{\eqref{eq:upperBoundDp}}{\leq} \delta+\tfrac{1}{n} < 1.
  \end{eqnarray*}
  For $n \geq\tfrac{2}{1-\delta}$, we have $\delta+\tfrac{1}{n} \leq \tfrac{1+\delta}{2} < 1$ which yields~\eqref{eq:TMin}. For $n \geq \max(\tfrac{2}{1-\delta},1/\delta)$, we have $\delta+\tfrac{1}{n} \leq \min(\tfrac{1+\delta}{2},2\delta) < 1$ which yields~\eqref{eq:TMin2}.
      
  \item {\bf Step 7.} To establish~\eqref{eq:lowerBoundDp} we  start with the expression~\eqref{eq:CpDevelopedExpression} in Cartesian coordinates, 
  \begin{align}
    \nonumber
    &-\tfrac{t}{2} D_p'(t)
    = -\tfrac{1}{n} q(t) \left( (n+1) q(t) - \int_{\R^n} t \dist(\vy / t, \mathcal{C}_1) \norm{\vy}_2^2 p_\vh(\vy) \di \vy  \right) \\
    &\Longrightarrow -\tfrac{t}{2} D_p'(t) + \tfrac{n+1}{n} q^{2}(t)    = \frac{q(t)}{n} \E \{ \dist(\vh, \mathcal{C}_t) \norm{\vh}_2^2\}. 
     \label{eq:lbnd_Cp}
  \end{align}
  Since $\tfrac{n+1}{n} q^2(t) = (n+1)D_p(t)$ we can rewrite \eqref{eq:lbnd_Cp} using the definition of $D_p(t)$ as
  \begin{equation}
    \label{eq:Dp_plus_Cp}
    D_p(t) -\tfrac{t}{2} D_p'(t)  = \frac{q(t)}{n} \E[\dist(\vh, {\cal C}_t) (\norm{\vh}_2^2 - n)].
  \end{equation}
  Observing further that $\E [\norm{\vh}_2^2 - n] = 0$ and $\E [\dist(\vh, {\cal C}_t)] = q(t)$, the following holds:
  \begin{eqnarray*}
    D_p(t) -\tfrac{t}{2} D_p'(t) 
    & = & \frac{q(t)}{n}\ \E\left\{(\dist(\vh, {\cal C}_t) - q(t)) (\norm{\vh}_2^2 - n)\right\} \\
    & \leq & \sqrt{D_{p}(t)}\ \E\left\{ \bigg|\dist(\vh, {\cal C}_t) - q(t)\bigg| \abs{\tfrac{\norm{\vh}_2^2 - n}{\sqrt{n}}} \right\} \\
    & = & \sqrt{D_{p}(t)}\  \int_0^{\infty} \prob \left\{ \bigg|\dist(\vh, {\cal C}_t) - q(t)\bigg| \abs{\tfrac{\norm{\vh}_2^2 - n}{\sqrt{n}}} \geq \epsilon \right\} \di \epsilon. \\
  \end{eqnarray*}
  The integrand can be controlled by a union bound as
  \[
    \prob \left\{ \bigg|\dist(\vh, {\cal C}_t) - q(t)\bigg| \abs{\tfrac{\norm{\vh}_2^2 - n}{\sqrt{n}}} \! \geq \! \epsilon \right\}
    \leq 
    \prob \left\{ \bigg|\dist(\vh, {\cal C}_t) \! - \! q(t)\bigg| \geq \sqrt{\epsilon} \right\}
    +
    \prob \left\{ \abs{\tfrac{\norm{\vh}_2^2 - n}{\sqrt{n}}} \geq \sqrt{\epsilon} \right\}
  \]
  which together with Lemma \ref{lem:concentration_results}\ref{lemitem:lipgausabs} for the first term and Lemma \ref{lem:concentration_results}\ref{lemitem:norm2abs} for the second term yields $D_p(t) -\tfrac{t}{2} D_p'(t) \leq C_{n} \sqrt{D_{p}(t)}$. \eqref{eq:lowerBoundDp} follows by definition of $t_p^*$ and 
  \begin{eqnarray*}
    C_{n}  \bydef   \int_{0}^\infty 2\left(\e^{-\epsilon / 2} +
    \begin{cases}
      \e^{-\frac{\epsilon}{8}} & \text{for $0 \leq \epsilon \leq n$}, \\
      \e^{-\frac{\sqrt{\epsilon n}}{8}} &\text{for $\epsilon > n$}
    \end{cases}
    \right) \di \epsilon
        &\leq& \int_{0}^\infty 2\left(\e^{-\epsilon / 2} + \e^{-\frac{\epsilon}{8}} + \e^{-\frac{\sqrt{\epsilon}}{8}}  
        \right) \di \epsilon 
        \bydef C < \infty.
  \end{eqnarray*}
\item {\bf Step 8. (Proof of Lemma~\ref{le:boundp2})}. Since $\dist(\vh,\norm{\cdot}_{2} \leq t) = (\norm{\vh}_{2}-t)_{+}$, we have $D_{2}(t;n) = {\tfrac{1}{n}} \left(\E (\norm{\vh}_2-t)_+\right)^{2}$. With $d(t) \bydef \sqrt{D_{2}(t;n)} = \tfrac{1}{\sqrt{n}} \E (\norm{\vh}_2-t)_+$ we have
\(
d'(t) = -\tfrac{1}{\sqrt{n}} \E [ \ind{\norm{\vh}_2>t} ] = -\tfrac{1}{\sqrt{n}} \prob(\norm{\vh}_2>t)
\)
and $-(t/2)D'_{2}(t) = -t d(t) d'(t)$. With a change of variables $\tau = t/\sqrt{n}$, define $F(\tau;n)  \bydef  F_{1}(\tau;n)+F_{2}(\tau;n)$ with
\begin{eqnarray*}
F_{1}(\tau;n) & \bydef & D_{2}(\tau\sqrt{n};n)  =   \left\{\E \left(\tfrac{\norm{\vh}_2}{\sqrt{n}} -\tau\right)_+\right\}^{2},   \\
F_{2}(\tau;n) & \bydef & -\tfrac{\tau\sqrt{n}}{2}D'_{2}(\tau\sqrt{n};n) =  \E \left\{\tfrac{\norm{\vh}_2}{\sqrt{n}} -\tau\right\}_+ \cdot \tau \prob\left\{\tfrac{\norm{\vh}_2}{\sqrt{n}}>\tau\right\}.
\end{eqnarray*}
Since $\norm{\vh}_{2}$ concentrates around $\sqrt{n}$ for large $n$, we show below that for any $0<\tau<1$,
\begin{equation}
\label{eq:limF1F2}
\lim_{n \to \infty} F_{1}(\tau;n) \bydef F_1(\tau) = (1-\tau)^{2}
\ \ \text{and} \ \
\lim_{n \to \infty} F_{2}(\tau;n) \bydef F_2(\tau) = (1-\tau)  \tau.
\end{equation}
It follows that for $\delta \in (0,1)$ we have $\lim_{n \to \infty} F(1-\delta;n)=\delta$ and 
\begin{align*}
&\lim_{n \to \infty} \tfrac{t^{*}_{2}(\delta;n)}{\sqrt{n}} =  1-\delta,
&&\lim_{n \to \infty} -(t^{*}_{2}/2) D'_{2}(t^{*}_{2};n)  =  F_{2}(\delta) = (1-\delta)\delta,\\
&\lim_{n \to \infty }D(t^{*}_{2};n) =  F_{1}(\delta) = \delta^{2},
&&\lim_{n \to \infty}\alpha^{*}_{2}(\delta;n)  =  \sqrt{\tfrac{\delta^{2}}{\delta(\delta-\delta^{2})}} = 1/\sqrt{1-\delta}.
\end{align*}
which establishes~\eqref{eq:UnivLowBoundDprime2_and_AlphaD2}.

To prove \eqref{eq:limF1F2} we compute as follows for $0<\tau<1$ and $\epsilon>0$:
\begin{eqnarray*}
\E \left\{ {\norm{\vh}_2}/{\sqrt{n}}-\tau\right\}_+ 
&=& 
\E \left\{ \left({\norm{\vh}_2}/{\sqrt{n}}-\tau\right)_+ 
\ \Big|\ 
\abs{{\norm{\vh}_2}/{\sqrt{n}} - 1} > \epsilon \right\} 
\prob\left\{ \abs{{\norm{\vh}_2}/{\sqrt{n}} - 1} > \epsilon\right\}\\ 
& + & \E \left\{ ({\norm{\vh}_2}/{\sqrt{n}} -\tau)_+ 
\ \Big|\ 
\abs{{\norm{\vh}_2}/{\sqrt{n}} -1} \leq \epsilon \right\} 
\prob\left\{ \abs{{\norm{\vh}_2}/{\sqrt{n}} -1} \leq \epsilon\right\},\\
\prob\left\{{\norm{\vh}_2}/{\sqrt{n}}>\tau\right\}
&=& 
\prob\left\{{\norm{\vh}_2}/{\sqrt{n}}>\tau 
\ \Big|\ 
\abs{{\norm{\vh}_2}/{\sqrt{n}} - 1} > \epsilon \right\}
\prob\left\{ \abs{{\norm{\vh}_2}/{\sqrt{n}} - 1} > \epsilon\right\}\\ 
&+ &\prob \left\{ {\norm{\vh}_2}/{\sqrt{n}} >\tau
\ \Big|\ 
\abs{{\norm{\vh}_2}/{\sqrt{n}} -1} \leq \epsilon \right\}
\prob\left\{ \abs{{\norm{\vh}_2}/{\sqrt{n}} -1} \leq \epsilon\right\}.
\end{eqnarray*}
For any $0<\epsilon < \min(1-\tau, \tfrac{1}{2})$ we get 
\begin{align*}
    \hspace{2mm} &\prob\left\{ \abs{{\norm{\vh}_2}/{\sqrt{n}} - 1}  \leq \epsilon\right\} \stackrel{(a)}{\geq} 1-2\e^{-c n \epsilon^2},
    && \hspace{-20mm} \prob \left\{{\norm{\vh}_2}/{\sqrt{n}}>\tau
    \ \Big|\ 
    \abs{{\norm{\vh}_2}/{\sqrt{n}}-1} \leq \epsilon
    \right\} = 1, \\
    &\E \left\{{\norm{\vh}_2}/{\sqrt{n}}-\tau\right\}_+ \geq  (1 - \tau-\epsilon) (1 - 2e^{-c n \epsilon^2}),
    && \hspace{-20mm} \prob\left\{{\norm{\vh}_2}/{\sqrt{n}}>\tau\right\}  \geq 1-2e^{-c n \epsilon^2}, \\
    &\E \left\{ \left({\norm{\vh}_2}/{\sqrt{n}}-\tau\right)_+ 
    \ \Big|\ 
    \abs{{\norm{\vh}_2}/{\sqrt{n}}-1} \leq \epsilon
    \right\}
    \geq  1-\tau-\epsilon > 0,
\end{align*}
where $(a)$ follows from Lemma \ref{lem:concentration_results}\ref{lemitem:gausnormabs} by noting that for $0 < \epsilon \leq \tfrac{1}{2}$, 
$\sqrt{{1}/{(1 - \epsilon)}} \leq 1 + \epsilon$. Hence, with $\epsilon = (1-\tau)/n^{1/4}$, 
\begin{eqnarray}
F_{1}(\tau;n) & \geq & (1-\tau)^{2} (1-n^{-1/4})^{2} (1-2e^{-c (1-\tau)^{2} n^{1/2}})^{2}\label{eq:lowerBoundF1},\\
F_{2}(\tau;n) & \geq & \tau(1-\tau) (1-n^{-1/4}) (1-2e^{-c (1-\tau)^{2}n^{1/2}})^{2}.\label{eq:lowerBoundF2}
\end{eqnarray}

For an upper bound, let $c_n \bydef \tfrac{1}{\sqrt{n}} \E \norm{\vh}_2$. Since $\tfrac{n}{\sqrt{n+1}} \stackrel{\eqref{eq:lowerBoundExpNormH}}{\leq} \E \norm{\vh}_{2} \leq \sqrt{n}$ we have $1 \geq c_{n} \geq \sqrt{\tfrac{n}{n+1}} \geq 1 - \tfrac{1}{2n} \Rightarrow 0 \leq 2(1 - c_n) \leq 1/n$. By Jensen's inequality, for $0<\tau<1$,
\begin{eqnarray*}
\E \left\{\tfrac{\norm{\vh}_2}{\sqrt{n}}-\tau\right\}_+ 
&\leq &
\sqrt{\E \left\{\tfrac{\norm{\vh}_2}{\sqrt{n}}-\tau\right\}_+^2}
\leq
\sqrt{\E \left\{\tfrac{\norm{\vh}_2}{\sqrt{n}}-\tau\right\}^2}
=
\sqrt{1-2c_n \tau + \tau^2}\\
&=& \sqrt{ (1-\tau)^2+2(1-c_n)\tau}
 \leq  \sqrt{(1-\tau)^{2} + \tau/n} = (1-\tau) \sqrt{1+\tfrac{\tau}{(1-\tau)^{2}n}},
\end{eqnarray*}
so that 
\(
F_{1}(\tau;n) \leq (1-\tau)^{2} 
+\tau/n
\)
and
\(
F_{2}(\tau;n) \leq \tau(1-\tau) \sqrt{1+\tfrac{\tau}{(1-\tau)^{2}n}}
\leq \tau (1-\tau) \left(1+\tfrac{\tau}{(1-\tau)^{2}n}\right).
\)
Combining all of the above yields~\eqref{eq:limF1F2}.
  \item {\bf Step 9. (Proof of Lemma~\ref{le:boundp1})}.
By~\eqref{eq:TMax} and~\eqref{eq:TMin} we have for any $n \geq 2/(1-\delta)$
    \[
    0<t_{\min}(\delta) \bydef \theta^{-1}(\tfrac{1+\delta}{2}) \leq t^{*}_{1}(\delta;n) \leq 2\theta^{-1}(\delta) \bydef t_{\max}(\delta).
    \]    
    By the continuity and strict monotonicity of $\theta$,
  \(
  \displaystyle
  V(\delta) \bydef \inf_{t \in [t_{\min}(\delta),t_{\max}(\delta)]} 
  \left\{
  \theta(t)-\theta(2t)
  \right\} >0.
  \)
By the convexity of $D_{1}(t)$ with $h = t = t^{*}_{1}$ we get 
$D_{1}(2t^{*}_{1}) \geq D_{1}(t^{*}_{1}) + t^{*}_{1} D'_{1}(t^{*}_{1})$, hence
\begin{eqnarray*}
-t^{*}_{1}D'_{1}(t^{*}_{1};n) 
\geq D_{1}(t^{*}_{1};n)-D_{1}(2t^{*}_{1};n) 
\ \stackrel{\mathclap{\eqref{eq:UpperBoundDpBarDp} \&\eqref{eq:LowerBoundDpBarDp}}}{\geq} \ 
 \bar{D}_{1}(t^{*}_{1};n)-\bar{D}_{1}(2t^{*}_{1};n)-\tfrac{1}{n}
&\stackrel{\eqref{eq:barD1theta}}{=}& 
\theta(t^{*}_{1})-\theta(2t^{*}_{1}) -\tfrac{1}{n}\\
&\geq& V(\delta) -\tfrac{1}{n}.
\end{eqnarray*}
For $n \geq N(\delta) \bydef \max(2/V(\delta),2/(1-\delta))$ we obtain $-t^{*}_{1}D'_{1}(t^{*}_{1};n) \geq V(\delta)/2 \bydef \gamma(\delta) > 0$ which establishes~\eqref{eq:UnivLowBoundDprime1}.

By~\eqref{eq:UpperBoundDpBarDp}--\eqref{eq:barD1theta}--\eqref{eq:LowerBoundDpBarDp} we have $\theta(t)-1/n \leq D_{1}(t;n) \leq \theta(t)$ for all $t$ and $n$. Hence, the sequence of convex differentiable functions $\{D_{1}(\, \cdot \, ;n)\}_{n}$ converges uniformly to the convex and smooth function $\theta(t)$ which implies convergence of the derivatives,
$\lim_{n \to \infty} D'_{1}(t;n) = \theta'(t)$. As for $p=2$, this shows $\lim_{n \to \infty} \left\{D_{1}(t;n)-(t/2) D'_{1}(t;n)\right\} = \theta(t)-(t/2)\theta'(t)$.
By Lemma~\ref{lem:theta_explicit}, $\theta(t)-(t/2)\theta'(t) = \erfc(t / \sqrt{2})$ so the unique $t = t^*_{1}(\delta)$ such that $\theta(t)-(t/2)\theta'(t) = \delta$ is 
\(
t^{*}_{1}(\delta) = \sqrt{2}\cdot\erfc^{-1}(\delta).
\)
Reasoning as for the case $p=2$ we get that
\begin{align*}
    \lim_{n \to \infty} t^{*}_{1}(\delta;n) & =  t^{*}_{1}(\delta),
    &\lim_{n \to \infty} -\tfrac{t^{*}_{1}}{2} D'_{1}(t^{*}_{1};n) & =  -\tfrac{t^*_{1}(\delta)}{2} \theta'(t^*_{1}(\delta)) > 0,\\
    \lim_{n \to \infty} D_{1}(t^{*}_{1};n) & =  \theta(t^*_{1}(\delta)),
    &\lim_{n \to \infty} \alpha^{*}_{1}(\delta;n) & =  \sqrt{\frac{\theta(t^*_{1}(\delta))}{\delta(\delta-\theta(t^*_{1}(\delta)))}}.
\end{align*}
Since $\erfc(t^{*}_{1}/\sqrt{2}) = \delta$ we have $\theta(t^{*}_{1}) = \delta -\bigg(\sqrt{\tfrac{2}{\pi}} \e^{-\tfrac{(t^{*}_{1})^{2}}{2}} t^{*}_{1} - \delta (t^{*}_{1})^{2}\bigg)$ and thus
\begin{equation}
    \lim_{n \to \infty} \alpha_{1}^{*}(\delta;n) = \sqrt{\frac{\delta-(\delta-\theta(t^*_{1}))}{\delta(\delta-\theta(t^*_{1}))}} 
    = \sqrt{ \bigg( \sqrt{\tfrac{2}{\pi}} \e^{-\tfrac{(t_1^*)^2}{2}} t^{*}_{1} - \delta(t_1^*)^2\bigg)^{-1} - \tfrac{1}{\delta}} \ .
    \end{equation}
\end{itemize}
\end{proof}

\begin{lemma}[Deterministic properties of $\kappa$]\label{le:ArgMinKappa}
Let $m, n$, $1 \leq m<n$ be two integers, $\delta \bydef (m-1)/n$, $1 \leq p \leq \infty$, $\kappa(\alpha,\beta)$ defined in~\eqref{eq:DefKappa}, $D_{p}(t)$ defined in \eqref{eq:DefDp}.
The following hold:
\begin{enumerate}
\item The function $\kappa(\alpha,\beta)$ is convex--concave and proper on $[0,\infty) \times [0,\infty)$, hence the function 
\(
  \kappa(\alpha) \bydef \sup_{0 \leq \beta \leq 1} \kappa(\alpha,\beta)
\)
is convex on $[0,\infty)$, and for any $A>0$ the function
\(
  \underline{\kappa}_{A}(\beta) \bydef \inf_{0 \leq \alpha \leq A} \kappa(\alpha,\beta)
\)
is concave on $[0,\infty)$.

\item The scalar $t^{*} = t^{*}_{p}(\delta;n)$ (cf. Lemma~\ref{lem:Dp_det_prop}--Property 7) is well defined, with $D_{p}(t^{*}) < \delta$.
\item Define
\begin{equation}\label{eq:kappaminimizerexplicit}
\alpha^{*} = \alpha^{*}(\delta;n) \bydef \sqrt{\tfrac{D_{p}(t^{*})}{\delta(\delta-D_{p}(t^{*}))}} \ .
\end{equation} 
For $A > \alpha^{*}$ and $\lambda \leq t^{*}$ we have
\begin{equation}\label{eq:kappaminimizer}
\argmin_{\alpha : 0 \leq \alpha \leq A} \max_{\beta: 0 \leq \beta \leq 1} \kappa(\alpha,\beta) 
= \alpha^{*}.
\end{equation} 
The corresponding optimal $\beta$ is $\beta^{*} = \beta^{*}(\lambda,\delta;n) \bydef \lambda/t^{*}$.

\item For $A>\alpha^{*}$, $\lambda \leq t^{*}$, $0<\epsilon \leq \max(\alpha^{*},A-\alpha^{*})$ we have 
\begin{equation}\label{eq:DefOmega}
\inf_{\abs{\alpha-\alpha^{*}} \geq \epsilon} \kappa(\alpha)-\kappa(\alpha^{*}) 
= \inf_{\abs{\alpha-\alpha^{*}}=\epsilon} \kappa(\alpha)-\kappa(\alpha^{*}) \geq 
\omega(\epsilon) = \omega_{p}(\epsilon;n,\delta,\lambda) \bydef
\tfrac{\epsilon^2}{2} 
\tfrac{\lambda \delta/t^{*}}{(1 + \delta(\alpha^* + \epsilon)^2)^{3/2}}.
\end{equation}
For the considered range of $\lambda$ and $\epsilon$, we have $\omega(\epsilon) \leq 1/2$.
\end{enumerate}
\end{lemma}

\begin{proof}[Proof of Lemma~\ref{le:ArgMinKappa}]
\hfill
\begin{enumerate}
\item It is obvious that $\kappa$ is proper. Convexity in $\alpha$ is easy to check by computing the second derivative. Concavity in $\beta$ follows from the convexity of $\beta \mapsto \Delta_{p}(\beta;\vh,\lambda)$ which is a distance to a convex set \cite[Example 3.16]{Boyd:2004uz}, and the fact that the expectation $\Delta_{p}(\beta,\lambda)$ of a convex function is convex.
As a result, $\kappa(\alpha)$ is convex and $\underline{\kappa}(\beta)$ is concave.
\item We have $\delta \bydef (m - 1) / n < 1$, and by Lemma~\ref{lem:Dp_det_prop}, Property 4, $D_{p}(0) \geq \tfrac{n}{n+1}$. Because we consider the underdetermined case, $1 \leq m < n$, we have $n \geq 2$ and
\[
\delta \leq (n - 2) / n \leq n / (n + 1) \leq D_p(0).
\]

\item Since $\kappa(\alpha,\beta)$ is convex--concave and proper, and the constraint sets in~\eqref{eq:kappaminimizer} convex and compact, we can change the order of maximization and minimization \cite[Corollary 3.3]{Sion:1958jm}. 

\textbf{(Minimization over $\alpha$)}~For $\beta = 0$ and any $\lambda \geq 0$ we have $\Delta_{p}(\beta,\lambda) = 0$ hence $\underline{\kappa}_{A}(0) = \inf_{0 \leq \alpha \leq A} \kappa(\alpha,0) = 0$. For $\beta>0$, observing that $\Delta_{p}(\beta,\lambda) = \beta \sqrt{D_{p}(\lambda/\beta)}$ we rewrite 
\[
\kappa(\alpha,\beta) = \beta(\sqrt{\delta\alpha^{2}+1}-\alpha \sqrt{D_{p}(\lambda/\beta)}).
\]
With $D_{p}^{-1}(y) \bydef \inf \set{t: D_{p}(t) \leq y}$ and $D_{p}$ strictly decreasing, we have
\(
  D_{p}(\lambda/\beta) \leq \frac{A^{2}\delta^{2}}{1+\delta A^{2}}
\)
if and only if $0<\beta \leq \wt{\beta} \bydef \tfrac{\lambda}{D_{p}^{-1}\left(A^{2}\delta^{2}/(1+\delta A^{2})\right)}$. Since $\tfrac{A^{2}\delta^{2}}{1+\delta A^{2}} < \delta$, this implies that for $0<\beta \leq \wt{\beta}$ we can define
\begin{equation}
     \wt{\alpha}(\beta) \bydef \sqrt{\frac{D_p(\lambda/\beta)}{\delta(\delta - D_{p}(\lambda/\beta))}}.
\end{equation}
and check that $\wt{\alpha}(\beta) \leq A$. By studying the sign of $\parder{\kappa(\alpha,\beta)}{\alpha} = \beta\left(\tfrac{\delta\alpha}{\sqrt{\delta \alpha^{2}+1}}-\sqrt{D_{p}(\lambda/\beta)}\right)$, we get that $\alpha \mapsto\kappa(\alpha,\beta)$ has a unique minimizer on $[0,A]$ which is precisely $\alpha^*(\beta) = \wt{\alpha}(\beta)$. It follows that for $0<\beta \leq \wt{\beta}$ we have:
 \begin{equation}\label{eq:ArgMinAlphaGivenBeta}
   \underline{\kappa}_{A}(\beta) = \min_{0 \leq \alpha \leq A} \kappa(\alpha,\beta) = \kappa(\alpha^*(\beta), \beta) = \beta \sqrt{\frac{\delta - D_p(\lambda/\beta)}{\delta}}. 
\end{equation}

\textbf{(Maximization over $\beta$)}~The sign of $\underline{\kappa}'_{A}(\beta)$ is that of $\delta-g(\lambda/\beta)$ with $g(t) \bydef D_{p}(t)-\tfrac{t}{2}D'_{p}(t)$ as in Lemma~\ref{lem:Dp_det_prop}--Property~7. Hence, we have: $\underline{\kappa}'_{A}(\beta) >0$ if $t=\lambda/\beta>t^{*}$ (that is to say if $\beta < \beta^{*} \bydef \lambda/t^{*}$); $\underline{\kappa}'_{A}(\beta) <0$ if $\beta > \beta^{*}$; and $D_{p}(t^{*}) < \delta$.

$A>\alpha^{*}$ implies $D_{p}(t^{*}) < A^{2}\delta^{2}/(1+\delta A^{2})$, i.e., $\beta^{*} < \wt{\beta}$. Combined with the fact that $\underline{\kappa}(0) = 0$, this shows that the supremum of $\underline{\kappa}(\beta)$ over $[0,\wt{\beta}]$ is achieved uniquely at $\beta^{*}$. This also implies that $\underline{\kappa}(\beta)$ is strictly decreasing for $\beta^{*} < \beta \leq \wt{\beta}$. Being concave, $\underline{\kappa}_{A}(\beta)$ must be also strictly decreasing for $\beta \geq \wt{\beta}$, so the supremum over $[0,\infty)$ is indeed achieved at $\beta^{*}$.
Since $\lambda \leq t^{*}$ we further have $\beta^{*} \leq 1$ hence this is also the supremum over $\beta \in [0,1]$.

To summarize, the optimal $\beta$ is $\beta^{*} = \lambda/t^{*}$, and the corresponding optimal $\alpha$ is given as
\begin{equation}
    \alpha^*(\beta^*) = \wt{\alpha}(\lambda/t^{*}) = \sqrt{\frac{D_p(t^*)}{\delta(\delta - D_p(t^*))}}.
\end{equation}

\item The assumption $\epsilon \leq \max(\alpha^{*},A-\alpha^{*})$ ensures that the set $\set{\alpha: \abs{\alpha-\alpha^{*}} \geq \epsilon,\ 0 \leq \alpha \leq A}$ is not empty. Since $\kappa(\alpha)$ is convex on $[0,\infty)$ with its minimum at $\alpha^{*}$, we have
    \begin{align*}
       \omega(\epsilon)
       &\bydef \inf_{\substack{\alpha : \abs{\alpha-\alpha^*} \geq \epsilon\\ 0 \leq \alpha \leq A}} \kappa(\alpha)-\kappa(\alpha^{*})
       = \min_{\substack{\alpha : \abs{\alpha - \alpha^*} = \epsilon\\ 0 \leq \alpha \leq A}} \kappa(\alpha) - \kappa(\alpha^*).
    \end{align*}

    Since $A>\alpha^{*}$ and $\lambda \leq t^{*}$, we have $0 < \beta^* < 1$. The second derivative of $\alpha \mapsto \kappa(\alpha, \beta^*)$ with respect to $\alpha$ reads
    \(
        \parderr{\kappa(\alpha, \beta^*)}{\alpha} = \frac{\beta^* \delta}{(1 + \delta \alpha^2)^{3/2}} > 0.
    \)
   implying that on $[0,A] \cap [\alpha^* - \epsilon, \alpha^* + \epsilon]$ the function $\alpha \mapsto \kappa(\alpha, \beta^*)$ is strongly convex with strong convexity modulus $\frac{\beta^* \delta}{(1 + \delta(\alpha^* + \epsilon)^2)^{3/2}}$. Since $\alpha \mapsto \kappa(\alpha,\beta^{*})$ is minimum at $\alpha^{*}$, it holds that
    \(
        \kappa(\alpha^* \pm \epsilon, \beta^*) \geq \kappa(\alpha^*, \beta^*) + \frac{\epsilon^2}{2} \frac{\beta^* \delta}{(1 + \delta(\alpha^* + \epsilon)^2)^{3/2}}.
    \)
    Furthermore, from the definition of $\kappa(\alpha)$ and $\beta^*$, we have that $\kappa(\alpha) \geq \kappa(\alpha, \beta^*)$ for any $\alpha$, with equality for $\alpha = \alpha^*$. The claim therefore follows.
    
    Since $\lambda \leq t^{*}$ and $\epsilon \leq \alpha^{*}$ we have $\omega(\epsilon) \leq \tfrac{\delta (\alpha^{*})^{2}}{2(1+\delta (\alpha^{*})^{2})^{3/2}} = f\left((1+\delta(\alpha^{*})^{2})^{-1/2}\right)$ with $f(u) \bydef \tfrac{1}{2} (1/u^{2}-1)u^{3} = \tfrac{u-u^{3}}{2} \leq 1/2$ for $0 < u \leq 1$. Hence, $\omega(\epsilon) \leq 1/2$.
        \end{enumerate}
\end{proof}
Invoking a lemma from \cite{Hjort:1993tq} we show that $\argmin_\alpha \phi(\alpha)$ concentrates around $\argmin_{\alpha}
\kappa(\alpha)$.

\begin{lemma}[{\cite[Lemma 2]{Hjort:1993tq}}] \label{le:hjortpollard}
Let $f(t)$ be a random convex function on some
    open set $S \subset \R^p$, and let $t_f$ be (one of) its minimizer(s). Consider
    another function $g(t)$ (which we interpret as approximating $f$), such
    that it has a unique argmin $t_g$. Then for each $\epsilon > 0$, we have
    that:
    \begin{equation}
        \label{eq:hjortpollard}
        \prob \{ \norm{t_f - t_g}_2 \geq \epsilon \} \leq \prob \left\{\sup_{\norm{s - t_g}_2
        \leq \epsilon} \abs{f(s) - g(s)} \geq \tfrac{1}{2} \inf_{\norm{s - t_g}_2 = \epsilon} (g(s) - g(t_g)) \right\}.
    \end{equation}
\end{lemma}
The role of $f(t)$ and $t_{f}$ will be played by $\phi(\alpha)$ and $\alpha^{*}_{\phi}$; the role of $g(t)$ and $t_{g}$ by $\kappa(\alpha)$ and $\alpha^{*}$.
\begin{lemma}\label{le:empiricalminmax}
        Let $A > \alpha^{*}$, $\lambda \leq t^{*}$, with $\alpha^{*}$ defined as in Lemma~\ref{le:ArgMinKappa}--Equation~\eqref{eq:kappaminimizer} and $t^{*} = t^{*}_{p}(\delta;n)$ as in Lemma~\ref{lem:Dp_det_prop}--Property 7. Consider the random function $\phi(\alpha)$ defined as in Corollary~\ref{cor:empiricalmaxbeta}, and
        \(
        \alpha^*_\phi \bydef \argmin_{0 \leq \alpha \leq A} \phi(\alpha).
        \) 
        For $0<\epsilon\leq \max(\alpha^{*},A-\alpha^{*})$ 
       we have
        \begin{equation}
            \prob \left[ \abs{\alpha^*_\phi - \alpha^*} \geq \epsilon \right] 
            \leq\ \zeta\big(n,\tfrac{\omega(\epsilon)}{2}\big),
        \end{equation}
        with $\zeta(n,\epsilon) = \zeta(n,\epsilon;A,\delta)$ defined in Lemma \ref{lem:concentration2} and $\omega(\epsilon)$ defined in~\eqref{eq:DefOmega}.
\end{lemma}

\begin{proof}
Since $0<\epsilon\leq \max(\alpha^{*},A-\alpha^{*})$ the set $\set{\alpha: 0 \leq \alpha \leq A,\ \abs{\alpha-\alpha^{*}} \leq \epsilon}$ is a non-empty subset of $[0,\ A]$, and 
\(
\sup_{\alpha: 0 \leq \alpha \leq A,\  \abs{\alpha - \alpha^{*}} \leq \epsilon} \abs{\phi(\alpha) - \kappa(\alpha)}
\leq 
\sup_{0 \leq \alpha \leq A} \abs{\phi(\alpha) - \kappa(\alpha)}.
\)
Since $\phi$ is a random convex function, we can apply Lemma~\ref{le:hjortpollard} to obtain
\begin{eqnarray*}
\prob \left\{ \abs{\alpha^{*}_{\phi}-\alpha^{*}} \geq \epsilon\right\}
 \leq 
\prob \left\{ \sup_{\alpha: 0 \leq \alpha \leq A,\ \abs{\alpha - \alpha^{*}} \leq \epsilon} \abs{\phi(\alpha) - \kappa(\alpha)} \geq  
\tfrac{\omega(\epsilon)}{2} \right\}
    & \leq &
    \prob \left\{ \sup_{0 \leq \alpha \leq A} \abs{\phi(\alpha) - \kappa(\alpha)} \geq \tfrac{\omega(\epsilon)}{2} \right\}\\
    & \leq & \zeta(n,\omega(\epsilon)/2),
\end{eqnarray*}
where we used that $\omega(\epsilon) \leq \inf_{\alpha: 0 \leq \alpha \leq A,\ \abs{\alpha-\alpha^{*}}=\epsilon} \kappa(\alpha)-\kappa(\alpha^{*})$, and the last inequality follows from Corollary~\ref{cor:empiricalmaxbeta} which we can use since $\omega(\epsilon)/2 \leq 1/4 < 2$. 
\end{proof}
  
\begin{lemma}
    \label{lem:phi_eps_separation}
    Let $A > \alpha^{*}$, $\lambda \leq t^{*}$ with $\alpha^{*}$ defined as in Lemma~\ref{le:ArgMinKappa}--Equation~\eqref{eq:kappaminimizer} and $t^{*} = t^{*}_{p}(\delta;n)$ as in Lemma~\ref{lem:Dp_det_prop}--Property 7. 
    For $0 \leq \epsilon\leq \max(\alpha^{*},A-\alpha^{*})$, consider the optimal cost of the auxiliary optimization~\eqref{eq:aux_dual_swap} with an altered order of minimization and
    maximization and $\norm{\vz}_{2}$ further restricted to be at least at distance $\epsilon$ from $\alpha^{*}$
        \begin{equation}
        \label{eq:opt_for_phi_eps}
        \phi_{\epsilon}(\vg, \vh) \bydef \max_{\substack{\beta : 0 \leq \beta \leq 1\\\vw:\norm{\vw}_{p^*} \leq 1}} \min_{\substack{\vz : \norm{\vz}_2 \leq A\\\norm{\vz}_2 \notin (\alpha^* - \epsilon, \alpha^* + \epsilon)}}  \max_{\vu:\norm{\vu}_2 = \beta} \tfrac{1}{\sqrt{n}} \norm{\vz}_2 \vg^\T \vu - \tfrac{1}{\sqrt{n}} \norm{\vu}_2 \vh^\T \vz 
    - \vu^\T \ve_1 + \tfrac{\lambda}{\sqrt{n}} \vw^\T \vz.
    \end{equation}
 With $\vg \sim \mathcal{N}(\vzero, \mI_m)$, $\vh \sim \mathcal{N}(\vzero, \mI_n)$, $1 \leq m < n$, $\delta = (m-1)/n$ we have
for any $0<\eta<2$: 
\begin{equation}
  \phi_{0}(\vg,\vh)  <  \kappa(\alpha^{*}) + \eta \quad \quad \text{and} \quad \quad
  \phi_\epsilon(\vg, \vh)  >  \kappa(\alpha^{*}) + \omega(\epsilon) - \eta,
\end{equation}
with probability at least $1-\zeta(n, \eta)$, where $\zeta(n,\epsilon) = \zeta(n,\epsilon;A,\delta)$ from~Lemma \ref{lem:concentration2}, $\omega(\epsilon)$ from~\eqref{eq:DefOmega}.
\end{lemma}

\begin{proof}
Since $0<\epsilon\leq \max(\alpha^{*},A-\alpha^{*})$ the set $S_{\epsilon} \bydef \set{\alpha: 0 \leq \alpha \leq A,\ \abs{\alpha-\alpha^{*}} \leq \epsilon}$ is a non-empty subset of $[0,\ A]$. Denote  $S^-_{\epsilon} \bydef [0, \alpha^* - \epsilon]$ and $S^+_{\epsilon} = [\alpha^* + \epsilon, A]$ its two convex components (at most one of them may be empty).
Let $\phi_{S^+_{\epsilon}}(\vg, \vh)$ the value of
\eqref{eq:opt_for_phi_eps}, but with $\norm{\vz}_2$ constrained to lie in $S^{+}_{\epsilon}$ (by convention, this is $+\infty$ when $S_{\epsilon}^{+} = \emptyset$). Similarly define $\phi_{S^-_{\epsilon}}(\vg, \vh)$. When $S^+_{\epsilon}$ is non-empty, since it is convex, we can
effect the same simplifications and min-max swaps as in the proof of Lemma
\ref{lem:concentration_column} (from \eqref{eq:aux_dual_swap_first_line} to \eqref{eq:opt_two_scalars}) to arrive at
\begin{equation*}
  \phi_{S^+_{\epsilon}}(\vg,\vh) = \min_{\alpha : \alpha^* + \epsilon \leq \alpha \leq A} \max_{\beta:0 \leq \beta \leq 1} \phi(\alpha, \beta; \vg, \vh),
\end{equation*}
and similary with $S^{-}_{\epsilon}$ we get when it is non-empty that
\begin{equation*}
  \phi_{S^-_{\epsilon}}(\vg,\vh) = \min_{\alpha : 0 \leq \alpha \leq \alpha^* - \epsilon} \max_{\beta:0 \leq \beta \leq 1} \phi(\alpha, \beta; \vg, \vh).
\end{equation*}
This shows that $\phi_{\epsilon}(\vg,\vh) = \min(\phi_{S^-_{\epsilon}}(\vg,\vh),\phi_{S^+_{\epsilon}}(\vg,\vh)) = \phi_{S_{\epsilon}}(\vg,\vh)$ where the notation $\phi_{S_{\epsilon}}(\vg,\vh)$ matches that used in Corollary~\ref{cor:empiricalmaxbeta}.
Moreover by definition (see Lemma~\ref{le:ArgMinKappa}) we have 
\begin{equation*}
    \kappa_{S_\epsilon} \bydef \min_{\alpha \in S_{\epsilon}} \max_{0 \leq \beta \leq 1} \kappa(\alpha, \beta) \geq \kappa(\alpha^{*}) + \omega(\epsilon).
\end{equation*} 
By  Corollary \ref{cor:empiricalmaxbeta} we have, for $0<\eta<2$, with probability at least $1-\zeta(n, \eta)$: for all $S \subset [0,A]$, $\abs{\phi_{S}(\vg,\vh)-\kappa_{S}} < \eta$. Specializing to $S = [0,A]$ and $S = S_{\epsilon}$ and combining the above yields 
\begin{align*}
    \phi_{0}(\vg,\vh) & = \phi_{[0,A]}(\vg,\vh) < \kappa_{[0,A]} + \eta = \kappa(\alpha^{*}) + \eta,\\
    \phi_{\epsilon}(\vg,\vh) & = \phi_{S_{\epsilon}}(\vg,\vh) > \kappa_{S_{\epsilon}}-\eta \geq \kappa(\alpha^*) + \omega(\epsilon) - \eta.
\end{align*}
\end{proof}

\begin{lemma}
    \label{lem:minmaxswap}
    Let $K > \alpha^{*}/\sqrt{n}$, $\lambda \leq t^{*}$ with $\alpha^{*}$ defined as in Lemma~\ref{le:ArgMinKappa}--Equation~\eqref{eq:kappaminimizer} and $t^{*} = t^{*}_{p}(\delta;n)$ as in Lemma~\ref{lem:Dp_det_prop}--Property 7.  Denote by $\wt{\vx}_K$ any optimal solution of \eqref{eq:po_bounded}.
    For $0 < \epsilon\leq \max(\alpha^{*},K\sqrt{n}-\alpha^{*})$ we have
    \[
        \prob \bigg\{ \abs{\norm{\wt{\vx}_K} - \tfrac{\alpha^*}{\sqrt{n}}} \geq \tfrac{\epsilon}{\sqrt{n}}  \bigg\}  \leq 4 \zeta\big(n, \tfrac{\omega(\epsilon)}{2}\big).
    \]
    with $\zeta(n,\epsilon)=\zeta(n,\epsilon;K\sqrt{n},\delta)$ defined in Lemma \ref{lem:concentration2} and $\omega(\epsilon)$ defined in~\eqref{eq:DefOmega}.
\end{lemma}
\begin{proof}
    Denote $\Phi$ the optimal cost of
    \eqref{eq:po_bounded} and $\Phi_\epsilon$ the corresponding cost when
    $\vz = \vx\sqrt{n}$ is further restricted to $\setS_{\epsilon} \bydef \set{\vz \ : \norm{\vz}_{2} \leq A\ \text{and}\ \abs{\norm{\vz}_2 - \alpha^\star} \geq \epsilon}$, with $A \bydef K\sqrt{n}$:
    \begin{equation}
 \label{eq:value_po_bounded}
  \Phi_{\epsilon} = \min_{\substack{\vx : \norm{\vx}_2 \leq K,\\ \abs{\sqrt{n}\norm{\vx}_2-\alpha^{*}} \geq \epsilon}} \max_{\vu: \norm{\vu}_2 \leq 1} \vu^\T \mA
 \vx + \lambda \norm{\vx}_p - \vu^\T \ve_1.
\end{equation}
    We now want to show that for $\epsilon>0$ we have with high probability $\Phi_{\epsilon} > \Phi = \Phi_{0}$, because this is equivalent
    to $\wt{\vz}_K = \wt{\vx}_K \sqrt{n} \in \setS'_{\epsilon} \bydef \set{\vz \ : \norm{\vz}_{2} \leq A\ \text{and}\ \abs{\norm{\vz}_2 - \alpha^*} < \epsilon}$.
    
    By Theorem~\ref{thm:cgmt}~\ref{thmpart:1}~and~\ref{thmpart:cgmt_optval_concentrate}, denoting $\phi^{P}_{\epsilon}$ (resp. $\phi^{D}_{\epsilon}$) the optimum value of the ``primal'' (resp. ``dual'') auxiliary optimization problem associated to~\eqref{eq:value_po_bounded},
       we have for any $c \in \R$ 
    \begin{equation}\label{eq:minmaxswap_step1}
        \prob \{\Phi_{\epsilon} < c\} \leq 2 \prob \{\phi^P_{\epsilon} \leq c\} 
        \quad \text{and} \quad
        \prob \{\Phi_{0} > c\} \leq 2 \prob \{\phi^D_{0} \geq c\},
    \end{equation}
where we additionally used that $\phi^D_{0} = \phi^P_{0}$ since we optimize over convex sets ($\setS_{0}$ is a convex ball) and the penalty $\psi(\vx,\vu) = \lambda \norm{\vx}_{p}-\vu^{\T}\ve_{1}$ is convex--concave (see, e.g., \cite[Corollary 3.3]{Sion:1958jm}).

Let $C(\vz, \vu, \vw) = C(\vz,\vu,\vz; \vg,\vh)$ be the objective function in
    \eqref{eq:aux_dual_swap} and~\eqref{eq:opt_for_phi_eps}, so that with $A \bydef K\sqrt{n}$ \eqref{eq:opt_for_phi_eps} becomes  
    \(
        \displaystyle
        \phi_{\epsilon} = \max_{\substack{\beta: 0 \leq \beta \leq 1\\\vw: \norm{\vw}_{p^*} \leq 1}} \min_{\substack{\vz : \norm{\vz}_2 \leq A\\ \vz \in \setS_{\epsilon}}} \max_{\vu: \norm{\vu}_2 = \beta} C(\vz, \vu, \vw),
    \)
    and the optimal cost of \eqref{eq:aux_dual_swap} reads $\phi = \phi_{0}$. With these notations, we have (since $\max \min \leq \min \max$ is always true):
    \begin{equation}
    \begin{aligned}
        \phi^P_{\epsilon} 
        &\bydef \min_{\substack{\vz: \norm{\vz}_2 \leq A\\\vz \in \setS_{\epsilon}}} \max_{\substack{\vu: \norm{\vu}_2 \leq 1, \\ \vw: \norm{\vw}_{p^*} \leq 1}} C(\vz, \vu, \vw) 
        = \min_{\substack{\vz: \norm{\vz}_2 \leq A\\\vz \in \setS_{\epsilon}}} \max_{\substack{\beta: 0 \leq \beta \leq 1, \\ \vw: \norm{\vw}_{p^*} \leq 1}} \max_{\vu: \norm{\vu}_2 = \beta} C(\vz, \vu, \vw) \\
        &\geq
        \max_{\substack{\beta: 0 \leq \beta \leq 1, \\ \vw: \norm{\vw}_{p^*} \leq 1}} \min_{\substack{\vz: \norm{\vz}_2 \leq A\\\vz \in \setS_{\epsilon}}} \max_{\vu: \norm{\vu}_2 = \beta} C(\vz, \vu, \vw) \ = \phi_{\epsilon}
    \end{aligned}
    \end{equation}
 and
    \begin{equation}
    \begin{aligned}
        \phi^D_{0}
        & \bydef  \max_{\substack{\vu: \norm{\vu}_2 \leq 1, \\ \vw: \norm{\vw}_{p^*} \leq 1}} \min_{\vz: \norm{\vz}_2 \leq A} C(\vz, \vu, \vw)
        = \max_{\substack{\beta: 0 \leq \beta \leq 1, \\ \vw: \norm{\vw}_{p^*} \leq 1}} \max_{\vu: \norm{\vu}_2 = \beta} \min_{\norm{\vz}_2 \leq A} C(\vz, \vu, \vw) \\
        &\leq
        \max_{\substack{\beta: 0 \leq \beta \leq 1, \\ \vw: \norm{\vw}_{p^*} \leq 1}} \min_{\vz: \norm{\vz}_2 \leq A} \max_{\vu: \norm{\vu}_2 = \beta} C(\vz, \vu, \vw) \ = \phi_{0}.
    \end{aligned}
    \end{equation}
    Denote $\underline{\phi}_{\epsilon} \bydef \kappa(\alpha^{*})+\omega(\epsilon)$ and $\bar{\phi} \bydef \kappa(\alpha^{*})$ and use the above with $c_{1} =  \underline{\phi}_{\epsilon}- \eta$ and $c_{2} =  \bar{\phi} + \eta$ where $\eta>0$ is arbitrary to get  
    \begin{equation}
    \begin{aligned}
        &\prob\{\Phi_{\epsilon} 
        <  \underline{\phi}_{\epsilon} - \eta\}
        &\leq &\ 2 \prob\{\phi^P_{\epsilon} \leq  \underline{\phi}_{\epsilon} - \eta\}
        &\leq&\ 2 \prob\{\phi_{\epsilon} \leq  \underline{\phi}_{\epsilon} - \eta\},\\
        &\prob\{\Phi > \bar{\phi} + \eta\}
        &\leq &\ 2 \prob\{\phi^D \geq \bar{\phi} + \eta\}
        &\leq&\ 2 \prob\{\phi_{0} \geq \bar{\phi} + \eta\}.
    \end{aligned}
    \end{equation}
    Consider the event ${\cal E} = \set{\Phi_{\epsilon} \geq
    \underline{\phi}_{\epsilon} - \eta \text{ and } \Phi \leq \bar{\phi} + \eta}$.
     For $0< \eta < (\underline{\phi}_{\epsilon}-\bar{\phi})/2 = \omega(\epsilon)/2$ we have $c_{1} > c_{2}$ hence this
    event implies that $\wt{\vz}_{K} \in \setS'_{\epsilon}$, which is what we wanted to prove. For such $\eta$, since $\omega(\epsilon)/2 \leq 1/4 <2$ we can use Lemma~\ref{lem:phi_eps_separation} and a union bound to obtain that this event happens with probability at least $1 - 4 \zeta(n, \eta)$. Hence, for any $0<\eta<\omega(\epsilon)/2$ we have
    \(
    \prob(\wt{\vz}_{K} \notin \setS_{\epsilon}) \leq 4\zeta(n,\eta).
    \)   
    By continuity of $\eta \mapsto \zeta(n,\eta)$ we take the limit when $\eta$ tends to $\omega(\epsilon)/2$. 
\end{proof}

\begin{lemma}
    \label{lemma:bounded_equals_unbounded}
    Let $K > \alpha^{*}/\sqrt{n}$, $\lambda \leq t^{*}$  with $\alpha^{*}$ defined as in Lemma~\ref{le:ArgMinKappa}--Equation~\eqref{eq:kappaminimizer} and $t^{*} = t^{*}_{p}(\delta;n)$ as in Lemma~\ref{lem:Dp_det_prop}--Property 7.  Denote by $\wt{\vx}_K$ any optimal solution of the random bounded problem \eqref{eq:po_bounded} and $\wt{\vx}$ any optimal solution of the random unbounded problem~\eqref{eq:lasso}.
    For $0 < \epsilon\leq \min(\alpha^{*},K\sqrt{n}-\alpha^{*})$ 
    we have 
    \[
      \prob \{ \wt{\vx}_K \neq \wt{\vx}\} \leq 
      4 \zeta\big(n, \tfrac{\omega(\epsilon)}{2}\big)
      \quad \text{and} \quad
      \prob \bigg\{ \abs{\norm{\wt{\vx}} - \tfrac{\alpha^*}{\sqrt{n}}} \geq \tfrac{\epsilon}{\sqrt{n}}  \bigg\}  \leq 4 \zeta\big(n, \tfrac{\omega(\epsilon)}{2}\big).
    \]
      with $\zeta(n,\epsilon)=\zeta(n,\epsilon;K\sqrt{n},\delta)$ defined in Lemma \ref{lem:concentration2} and $\omega(\epsilon)$ defined in~\eqref{eq:DefOmega}.
\end{lemma}

\begin{proof}
To handle non-uniqueness, $\wt{\vx}$ (resp. $\wt{\vx}_{K}$) may denote the convex set of solutions of the respective convex optimization problems. The property $\wt{\vx} \neq \wt{\vx}_K$ then means that the sets do not intersect, and inequalities such as $f(\wt{\vx}) > c$ are meant to hold for all elements of the set $\wt{\vx}$. 

We first prove, by contradiction, that if $\wt{\vx} \neq \wt{\vx}_K$, then necessarily $\norm{\wt{\vx}_{K}}_2 = K$.
Suppose that the opposite holds: $\wt{\vx} \neq \wt{\vx}_K$, but $\norm{\wt{\vx}_{K}}_2 < K$. Since $\wt{\vx} \neq \wt{\vx}_{K}$ we have $\norm{\wt{\vx}}_2 > K$. Denoting the objective in \eqref{eq:lasso} by $\phi(\vx)$, this means that $\phi(\wt{\vx}) < \phi(\wt{\vx}_K)$. By convexity of $\phi$ it follows that all points on the line segment $\vx_\nu = \nu \wt{\vx}_K + (1 - \nu) \wt{\vx}$, $0\leq \nu \leq 1$ satisfy 
    \begin{equation}
        \label{eq:lemma_bounded_is_ok_cvx}
        \phi(\vx_\nu) \leq \phi(\wt{\vx}_K).         
    \end{equation}   
Since $\norm{\vx_{0}}_2 > K$ and $\norm{\vx_{1}}_2 < K$, by continuity there exists $\nu \in (0,1)$ such that $\norm{\vx_\nu}_2 = K$. Further, by \eqref{eq:lemma_bounded_is_ok_cvx}, $\vx_\nu$ is optimizing the bounded problem \eqref{eq:po_bounded}, contradicting our assumption. 

By contraposition, if $\norm{\wt{\vx}_{K}}_2 < K$ then $\wt{\vx} = \wt{\vx}_{K}$. In particular, since we assume $K\sqrt{n}> \alpha^{*}+\epsilon$, we have: 
if  $\abs{\sqrt{n}\norm{\wt{\vx}_K}_2-\alpha^{*}} < \epsilon$ then $\norm{\wt{\vx}_{K}}_2 < K$ 
hence $\wt{\vx} = \wt{\vx}_{K}$ and $\abs{\sqrt{n}\norm{\wt{\vx}}_2-\alpha^{*}} = \abs{\sqrt{n}\norm{\wt{\vx}_K}_2-\alpha^{*}} < \epsilon$.
From here, it follows that $\prob\{ \wt{\vx}_K \neq \wt{\vx}\} \leq \prob \{ \abs{\sqrt{n}\norm{\wt{\vx}_K}_2-\alpha^{*}} \geq \epsilon\}$ and
\(
\prob \{ \abs{\sqrt{n}\norm{\wt{\vx}}_2-\alpha^{*}} \geq \epsilon\} 
\leq
\prob \{ \abs{\sqrt{n}\norm{\wt{\vx}_K}_2-\alpha^{*}} \geq \epsilon\}.
\)
We conclude using Lemma~\ref{lem:minmaxswap}.
\end{proof}

\begin{lemma}
    \label{lem:sqrtn_lipschitz}
    With $\Delta_{p}$ defined as in the statement of Lemma~\ref{lem:concentration2}, the function $f_{\vh}(\beta) \bydef \abs{
    \Delta_p(\beta; \vh,\lambda) - \Delta_p(\beta;n,\lambda)}$ is $\max \set{\norm{\vh}/\sqrt{n},1}$-Lipschitz in $\beta$.
\end{lemma}

\begin{proof}
    We omit the dependency in $\lambda$ for brevity and write by definition 
    \[
        \abs{\Delta_p(\beta_1; \vh) - \Delta_p(\beta_2; \vh)}
        = \tfrac{1}{\sqrt{n}}\ \abs{\dist(\beta_1 \vh, {\cal C}) - \dist(\beta_2 \vh, {\cal C})}
        \leq \tfrac{1}{\sqrt{n}} \ \norm{\beta_1 \vh - \beta_2 \vh}_2 \leq \tfrac{\norm{\vh}_2}{\sqrt{n}} \abs{\beta_1 - \beta_2},
    \]
    since the Euclidean distance to a convex set is 1-Lipschitz with respect to the Euclidean metric. Further, 
    \(
        \abs{\Delta_p(\beta_1) - \Delta_p(\beta_2)} 
        \leq \E \abs{ (\Delta_p(\beta_1; \vh)  - \Delta_p(\beta_2; \vh))}
        \leq \tfrac{\abs{\beta_1 - \beta_2}}{\sqrt{n}}\ \E \norm{\vh}_2
        \leq \abs{\beta_1 - \beta_2}.
    \)
    Now observe that the Lipschitz constant of the difference of two Lipschitz functions does not exceed the largest of the two constants, and the Lipschitz constant of $|f|$ equals that of $f$.
\end{proof}

\bibliographystyle{siamplain}

\bibliography{pseudo,pseudo_supp}

\begin{thebibliography}{10}

\bibitem{Amelunxen2014}
{\sc D.~Amelunxen, M.~Lotz, M.~B. McCoy, and J.~A. Tropp}, {\em Living on the
  edge: phase transitions in convex programs with random data}, Information and
  Inference: A Journal of the IMA, 3 (2014), pp.~224--294.

\bibitem{barvinok2005math}
{\sc A.~Barvinok}, {\em Math 710: Measure concentration}, Lecture notes,
  (2005), \url{http://www.math.lsa.umich.edu/~barvinok/total710.pdf}.

\bibitem{Boyd:2004uz}
{\sc S.~Boyd and L.~Vandenberghe}, {\em {Convex Optimization}}, Cambridge
  University Press, 2004.

\bibitem{Casazza:ev}
{\sc P.~G. Casazza, A.~Heinecke, F.~Krahmer, and G.~Kutyniok}, {\em {Optimally
  Sparse Frames}}, IEEE Trans. Inf. Theory, 57, pp.~7279--7287.

\bibitem{Chandrasekaran:2012fs}
{\sc V.~Chandrasekaran, P.~A. Parrilo, and A.~S. Willsky}, {\em {Convex Graph
  Invariants}}, SIAM Rev., 54 (2012), pp.~513--541.

\bibitem{Chiani:2003eu}
{\sc M.~Chiani, D.~Dardari, and M.~K. Simon}, {\em {New exponential bounds and
  approximations for the computation of error probability in fading channels}},
  IEEE Trans. Wirel. Commun., 24 (2003), pp.~840--845.

\bibitem{DMA97}
{\sc G.~Davis, S.~Mallat, and M.~Avellaneda}, {\em {Adaptive Greedy
  Approximations}}, Constr. Approx., 13 (1997), pp.~57--98.

\bibitem{beyondMP-partI}
{\sc I.~Dokmanic and R.~Gribonval}, {\em Beyond moore-penrose part {I:}
  generalized inverses that minimize matrix norms}, CoRR, abs/1706.08349
  (2017), \url{http://arxiv.org/abs/1706.08349},
  \url{https://arxiv.org/abs/1706.08349}.

\bibitem{beyondMP-partII}
{\sc I.~Dokmanic and R.~Gribonval}, {\em Beyond moore-penrose part {II:} the
  sparse pseudoinverse}, CoRR, abs/1706.08701 (2017),
  \url{http://arxiv.org/abs/1706.08701},
  \url{https://arxiv.org/abs/1706.08701}.

\bibitem{Dokmanic:2013bo}
{\sc I.~Dokmani{\'c}, M.~Kolund{\v z}ija, and M.~Vetterli}, {\em {Beyond
  Moore-Penrose: Sparse pseudoinverse}}, in IEEE ICASSP, IEEE, 2013,
  pp.~6526--6530.

\bibitem{Donoho4273}
{\sc D.~Donoho and J.~Tanner}, {\em Observed universality of phase transitions
  in high-dimensional geometry, with implications for modern data analysis and
  signal processing}, Philosophical Transactions of the Royal Society of London
  A: Mathematical, Physical and Engineering Sciences, 367 (2009),
  pp.~4273--4293.

\bibitem{Donoho:2006ci}
{\sc D.~L. Donoho}, {\em {Compressed Sensing}}, IEEE Trans. Inf. Theory, 52
  (2006), pp.~1289--1306.

\bibitem{Foygel:2014iy}
{\sc R.~Foygel and L.~Mackey}, {\em {Corrupted Sensing: Novel Guarantees for
  Separating Structured Signals}}, IEEE Trans. Inf. Theory, 60 (2014),
  pp.~1223--1247.

\bibitem{Gordon1988escape}
{\sc {Gordon, Y.}}, {\em {On Milman's inequality and random subspaces which
  escape through a mesh in $\mathbb{R}^n$}}, in Geometric Aspects of Functional
  Analysis, Berlin, Heidelberg, 1988, Springer Berlin Heidelberg, pp.~84--106.

\bibitem{Gordon1985}
{\sc {Gordon, Yehoram}}, {\em {Some inequalities for Gaussian processes and
  applications}}, {Israel Journal of Mathematics}, 50 (1985), pp.~265--289.

\bibitem{Halko:2011kg}
{\sc N.~Halko, P.~G. Martinsson, and J.~A. Tropp}, {\em {Finding Structure with
  Randomness: Probabilistic Algorithms for Constructing Approximate Matrix
  Decompositions}}, SIAM Rev., 53 (2011), pp.~217--288.

\bibitem{Hjort:1993tq}
{\sc N.~L. Hjort and D.~Pollard}, {\em {Asymptotics for Minimisers of Convex
  Processes}}, arXiv,  (2011), \url{https://arxiv.org/abs/1107.3806v1}.

\bibitem{Krahmer:2012cy}
{\sc F.~Krahmer, G.~Kutyniok, and J.~Lemvig}, {\em {Sparsity and Spectral
  Properties of Dual Frames}}, Linear Algebra Appl., 439 (2012), pp.~1--17.

\bibitem{Ledoux:1999ip}
{\sc M.~Ledoux}, {\em {Concentration of Measure and Logarithmic Sobolev
  Inequalities}}, in S{\'e}minaire de Probabilit{\'e}s XXXIII, Springer,
  Berlin, Heidelberg, 1999, pp.~120--216.

\bibitem{Li:2013cx}
{\sc S.~Li, Y.~Liu, and T.~Mi}, {\em {Sparse Dual Frames and Dual Gabor
  Functions of Minimal Time and Frequency Supports}}, J. Fourier Anal. Appl.,
  19 (2013), pp.~48--76.

\bibitem{natarajan95:_spars}
{\sc B.~Natarajan}, {\em {Sparse approximate solutions to linear systems}},
  SIAM J. Computing, 25 (1995), pp.~227--234.

\bibitem{Oymak:2013vm}
{\sc S.~Oymak, C.~Thrampoulidis, and B.~Hassibi}, {\em {The Squared-Error of
  Generalized LASSO: A Precise Analysis}}, arXiv,  (2013),
  \url{https://arxiv.org/abs/1311.0830v2}.

\bibitem{OymakTropp2015}
{\sc S.~Oymak and J.~A. Tropp}, {\em Universality laws for randomized dimension
  reduction, with applications}, Information and Inference: A Journal of the
  IMA,  (2017), pp.~1--110.

\bibitem{Perraudin:2014we}
{\sc N.~Perraudin, N.~Holighaus, P.~L. S{\o}ndergaard, and P.~Balazs}, {\em
  {Designing Gabor Windows Using Convex Optimization}}, arXiv,  (2014),
  \url{https://arxiv.org/abs/1401.6033}.

\bibitem{piot2015optical}
{\sc J.~Piot, M.~Kolund\v{z}ija, D.~Korchagin, I.~Dokmani\'{c}, M.~Vetterli,
  and O.~Drumm}, {\em Optical touch tomography}, July~28 2015.
\newblock US Patent 9,092,091.

\bibitem{Rudelson:2008gv}
{\sc M.~Rudelson}, {\em {Invertibility of random matrices: norm of the
  inverse}}, Ann. Math., 168 (2008), pp.~575--600.

\bibitem{RudelsonVershynin}
{\sc M.~Rudelson and R.~Vershynin}, {\em On sparse reconstruction from fourier
  and gaussian measurements}, Comm. Pure Appl. Math., 61 (2008),
  pp.~1025--1045.

\bibitem{Sion:1958jm}
{\sc M.~Sion}, {\em {On General Minimax Theorems}}, Pac. J. Math., 8 (1958),
  pp.~171--176.

\bibitem{Stojnic2009}
{\sc M.~Stojnic}, {\em Various thresholds for $\ell^1$-optimization in
  compressed sensing}, CoRR, abs/0907.3666 (2009),
  \url{http://arxiv.org/abs/0907.3666}, \url{https://arxiv.org/abs/0907.3666}.

\bibitem{Stojnic2013}
{\sc M.~Stojnic}, {\em A framework to characterize performance of {LASSO}
  algorithms}, CoRR, abs/1303.7291 (2013),
  \url{http://arxiv.org/abs/1303.7291}, \url{https://arxiv.org/abs/1303.7291}.

\bibitem{Thrampoulidis:2016vo}
{\sc C.~Thrampoulidis, E.~Abbasi, and B.~Hassibi}, {\em {Precise Error Analysis
  of Regularized M-estimators in High-dimensions}}, arXiv,  (2016),
  \url{https://arxiv.org/abs/1601.06233v1}.

\bibitem{Thrampoulidis:2015vf}
{\sc C.~Thrampoulidis, S.~Oymak, and B.~Hassibi}, {\em Regularized linear
  regression: A precise analysis of the estimation error}, in Proc. COLT,
  vol.~40, Paris, France, 03--06 Jul 2015, PMLR, pp.~1683--1709.

\bibitem{Vershynin:2009fv}
{\sc R.~Vershynin}, {\em {Introduction to the non-asymptotic analysis of random
  matrices}}, in Compressed Sensing, Y.~C. Eldar and G.~Kutyniok, eds.,
  Cambridge University Press, Cambridge, 2009, pp.~210--268.

\bibitem{Vershynin:2014jv}
{\sc R.~Vershynin}, {\em {Invertibility of symmetric random matrices}}, Random
  Struct. Algorithms, 44 (2014), pp.~135--182.

\bibitem{wainwright}
{\sc M.~Wainwright}, {\em Basic tail and concentration bounds, Draft}, 2015,
  \url{https://www.stat.berkeley.edu/~mjwain/stat210b/Chap2_TailBounds_Jan22_2015.pdf}.

\bibitem{Wendel:1948fv}
{\sc J.~G. Wendel}, {\em {Note on the Gamma Function}}, Am. Math. Mon., 55
  (1948), p.~563.

\end{thebibliography}

\end{document}